\documentclass[12pt,reqno]{amsart}
\usepackage[margin=1.4in]{geometry}
\usepackage{derivative}
\usepackage{tabularx}
\usepackage{amssymb} 

\usepackage{xcolor}
\usepackage{lmodern}
\usepackage{hyperref, amsmath, amssymb, mathtools}
\usepackage{verbatim}
\usepackage{amsthm}
\usepackage{amsfonts}
\usepackage{amsmath}
\usepackage{mathrsfs}  
\usepackage{comment}
\usepackage{chngcntr}
\usepackage{ctable}
\usepackage{url}
\usepackage{hyperref}
\usepackage[numbers]{natbib}
\newtheorem{thm}{Theorem}
\newtheorem*{thm*}{Theorem}
\newtheorem{lemma}[thm]{Lemma}
\newtheorem{corollary}[thm]{Corollary}

\newtheorem{prop}[thm]{Proposition}
\newtheorem*{prop*}{Proposition}
\theoremstyle{definition}

\newtheorem{defn}{Definition}
\newtheorem*{remark}{Remark}
\numberwithin{thm}{section}

\newcommand{\Z}{\mathbb{Z}}

\newcommand{\R}{\mathbb{R}}
\newcommand{\Q}{\mathbb{Q}}

\newcommand{\SL}{\textrm{SL}}
\newcommand{\GL}{\textrm{GL}}

\newcommand{\Spec}{\mathrm{Spec}}

\title{Bass note spectra of Binary forms}
\author{Giorgos Kotsovolis}
\address{Department of Mathematics, Princeton University, Princeton, NJ 08540}
\email{gk13@princeton.edu}
\date{\today}

\begin{document}

\maketitle

\begin{abstract}
  We show that the spectrum of every $\R-$isotropic homogeneous binary form $P$ of degree $n\geq\nolinebreak3$ is an interval of the form $[0,M_P],$ where $M_P$ is some positive constant. This completes the discussion around a conjecture of Mordell from 1940 (disproved by Davenport) regarding the existence of spectral gaps for binary cubic forms and further settles Mahler's program for binary forms of every degree.  
\end{abstract}
\tableofcontents
\section{Introduction}
Let $P\in\R\left[x_1,x_2,...,x_k\right]$ denote a homogeneous polynomial of degree $n\in\mathbb{N}$.
\begin{itemize}
    \item For a lattice $\Lambda$ in $\R^k$, what is the smallest value that $\vert P(v)\vert$ assumes as $v$ ranges over all non-trivial vectors of $\Lambda$? 
    \item What is the distribution of these smallest values as $\Lambda$ ranges over all $k-$dimensional unimodular lattices?
    \end{itemize}By posing these fundamental questions, Mahler (\cite{Mahler},\cite{mahlerII}) laid the foundations of a general theory of $k-$dimensional star bodies and their extremal lattices, which allows for a systematic approach to classical number-theoretical and dynamical problems (see Sarnak's Chern Lectures \cite{Sarnak}). However, although Mahler's program of understanding $$\Spec(P):=\left\{\inf_{\underline{x} \in\Lambda \setminus \underline{0}}\left\vert P(\underline{x})\right\vert, \Lambda \subset \R^k \text{ a unimodular lattice} \right\}$$ offers a unified framework for many questions in the Geometry of Numbers, the nature of this spectrum as we vary the form $P$ is diverse enough, so that one can not hope to give an answer in full generality. In this paper we address these questions for the case of real homogeneous binary forms; polynomials of the form $$P(x,y)=\sum_{i\geq0}^{n}c_ix^iy^{n-i},$$ for some $n\in\mathbb{N}$ and $c_i\in\R$ for all $i\in\{1,2,...,n\}$. \\ \\ To a binary form $P$ we can associate the differential operator $$T_P=P\left(\frac{\partial}{\partial x},\frac{\partial}{\partial y}\right)$$ acting on $C^{\infty}_0\left(\R^2/\Lambda\right)$ for every two dimensional unimodular lattice $\Lambda$. The smallest in absolute value eigenvalue of $T_P$ is $$\vert\lambda_1(\R^2/\Lambda)\vert=4\pi^2\inf_{\underline{x} \in\Lambda \setminus \underline{0}}\left\vert P(\underline{x})\right\vert,$$ and we see that there is a correspondence between the spectra of binary forms and the $\lambda_1$ eigenvalues of the operators $T_P$, when acting on tori of volume $1$. Under this viewpoint, Mahler's questions can be seen as instances of the Bass Note Spectra of differential operators on locally uniform geometries (see \cite{Sarnak}).  \\ \\A binary form $P$ will be called \textit{$\R$-anisotropic} if the only solution of the equation $$P(x,y)=0$$ is $x=y=0.$ If $P$ is not $\R$-anisotropic, we call it $\R$-\textit{isotropic}. If the form $P$ is $\R$-anisotropic, the intermediate value theorem implies that the spectrum is an interval and the question reduces to determining the extremal lattices, as shown by Mahler (see Section \ref{General Theory of Star Bodies}):
\begin{thm}[Mahler]\label{Positive definite}
    Let $P$ denote an $\R$-anisotropic homogeneous binary form of degree $n\geq 2$. Then $$\Spec(P)=(0,M_P],$$ for some $M_P>0.$
\end{thm} 
 However, when the form $P$ is $\R-$isotropic, the nature of $\Spec(P)$ can be intricate. The spectrum of a binary quadratic form of discriminant $1$ is equal to the well studied Markoff spectrum, which we denote by $\mathcal{M}$. \begin{thm}
 For the Markoff spectrum $\mathcal{M}$ we have: \begin{itemize}
     \item (Markoff \cite{Markov}) Let $m$ denote an integer solution of the equation $$x^2+y^2+z^2=3xyz.$$ Then $\frac{m}{\sqrt{9m^2-4}}\in\mathcal{M}.$ Furthermore, these are all the points of the spectrum larger than $\frac{1}{3}.$ 
     \item (Hall \cite{Hall}) $\mathcal{M}$ contains an interval of the form $[0,c]$.
     \item (Freiman \cite{Freiman}) The largest value of $c$ is equal to the constant $$c_{max}=\left(\frac{2221564096 + 283748\sqrt{462}}{491993569}\right)^{-1}.$$ \end{itemize}  
 \end{thm}  
For degrees $n\geq3$ much less is known.  The case of binary cubic forms was investigated by Mordell (\cite{Mordell}) who determined the extremal lattices for the star body problem and consequently the maximum point of $\Spec(P).$ \begin{thm}[Mordell]
Let $P$ denote a binary cubic form with discriminant $D_P\neq 0$. Then, $$\max\Spec(P) = \begin{cases}
  \sqrt[4]{\frac{-D_P}{23}}, & \text{if } D_P<0, \\[10pt] \sqrt[4]{\frac{D_P}{49}}, & \text{if } D_P>0.
\end{cases}$$
\end{thm}
Observing the existence of a spectral gap already for the case of indefinite quadratic forms, Mordell further suggested that this maximum point is an isolated point of the spectrum, a conjecture disproved by Davenport in \cite{DavenportonMordell} by constructing a sequence of infima over unimodular lattices converging to the maximum. \\

In this work, we give a complete description of $\Spec(P)$ for all binary cubic forms:

\begin{thm}\label{D<0}
Let $P$ be a binary cubic form with $D_P<0$. Then, $$\Spec(P)=\left[0,\sqrt[4]{\frac{-D_P}{23}}\right].$$
\end{thm}
Similarly:
\begin{thm}\label{D>0}

Let $P$ be a binary cubic form with $D_P>0$. Then, $$\Spec(P)=\left[0,\sqrt[4]{\frac{D_P}{49}}\right].$$
\end{thm}

We prove Theorems \ref{D<0} and \ref{D>0} by proving a stronger statement that the values of normalized infima of binary cubic forms fill the whole spectral interval, even when we restrict to lattices in an arbitrarily small neighborhood of the extremal lattice. \\

For a general higher degree form $P$, the spectrum is not always an interval. For example, for $P(x,y)=x^2y^2,$ we have $\Spec(P)=\left\{t^2\vert t\in \Spec(xy)\right\}$ which behaves similarly to the Markoff Spectrum. We can exclude simple counterexamples coming from spectra of lower degree forms by assuming $D_P \neq 0$. 
To the best of our knowledge, there has been no progress in understanding the nature of $\Spec(P)$ for $\R-$isotropic binary forms of higher degree. In this paper, we resolve this question for all binary forms of degree $n\geq 3$ and $D_P\neq 0$. Our main result is the following:

\begin{thm} \label{Main Theorem}
    Let $P(x,y)$ be an $\R-$isotropic binary form of degree $n\geq 3$ with non-zero discriminant. Then $\Spec(P)=[0,M_P]$, for some $M_P>0.$
\end{thm}

It follows that the case of indefinite binary quadratic forms is the only one where the spectrum is not an interval. Interestingly, for a homogeneous binary form $P$, the body $\{\underline{x}\in\R^2: \vert P(\underline{x})\vert\leq1\}$ is of finite volume if and only if $P$ is not quadratic indefinite. Although this fact will not be used explicitly in our proof, it seems to be the underlying reason for this phenomenon. \\

The organization of this paper is as follows:
\begin{itemize}

    \item In Section \ref{General Theory of Star Bodies}, we give an overview of Mahler's general theory of star bodies and prove existence theorems for extremal lattices. 
    \item In Section \ref{Preliminaries on Continued Fractions}, we review some well known facts regarding continued fractions, that we will be using freely throughout the paper.
    \item In Section \ref{The star body problem for binary cubic forms}, we address binary cubic forms. In subsection \ref{The case of negative discriminant} we give a short proof of Theorem \ref{D<0}. The proof of Theorem \ref{D>0} is substantially more convoluted. In subsection \ref{The case of positive discriminant} we prove that the spectrum of a binary cubic form of positive discriminant is dense inside the corresponding spectral interval.
    \item  In Section \ref{The sets}, we introduce and discuss the properties of certain Diophantine sets. 
    \item In Section \ref{The structural theorem and density of approximant lattices}, we prove that for any sufficiently thick set $X\subset\R^k$ and any binary form $P$ of degree $n\geq 3$ of signature $(k,n-k)$ and non-vanishing discriminant, we can find a form $P'$ in the same $\SL_2(\R)$-orbit such that the real roots of $P'(z,1)=0$ lie in $X$ and the infima of $P,P'$ on $\Z^2\backslash\underline{0}$ are arbitrarily close. This fact is already enough to deduce that the spectrum of binary forms has no isolated maximum for degrees $n\geq3.$
    \item In Section \ref{The Measure of Spec(P)}, we generate spectra with measure arbitrarily close to maximal using thin neighborhoods of diagonal perturbations and prove that $\Spec(P)$ is an interval of the form $[0,M_P]$ modulo a set of measure zero. 
\item In Section \ref{Fixed point perturbations and intervals in the spectrum}, we show how to pass from full measure to full interval by using fixed point perturbations inside certain curves of $\SL_2(\R)$, hence concluding the proof of Theorem \ref{Main Theorem}.
\end{itemize}
\section{General Theory of Star Bodies}\label{General Theory of Star Bodies}
Suppose $F$ is a non-negative continuous function on $\R^k$ and assume further that $F$ is homogeneous of degree $n$: $$F(t\underline{x})=\vert t\vert^nF(\underline{x}).$$

The \textit{star body} associated with $F$ is the set equal to $K=\{\underline{x}\in\R^k:\,F(\underline{x})\leq1\}$ and we call a lattice $\Lambda$ of $\R^k$ $K-$admissible if the only $v\in\Lambda\cap K$ is the trivial vector with all coordinates equal to zero. By definition, the determinant of the star body $K$ is equal to: $$\Delta(K)=\inf\{\sqrt[k]{\left\vert\det(\Lambda)\right\vert}:\text{  $\Lambda$ is a $K-$admissible lattice}\}$$  or equivalently  $$\Delta^{-1}(K)=\sup_{\Lambda\text{ unimodular}}\inf_{\underline{x}\in\Lambda\backslash\underline{0}}\vert F(\underline{x})\vert=\sup (\Spec(F)).$$

For a given $F$, the existence of $K-$admissible lattices is not guaranteed and by convention we set $\Delta(K)=\infty$ if no $K-$admissible lattices exist. From now on, we consider only star bodies of finite type ($\Delta(K)<\infty$). An extremal lattice for the star body $K$ is a $K-$admissible lattice $\Lambda$ such that $$\det(\Lambda)=\Delta^k(K).$$ As a consequence of his celebrated compactness theorem, Mahler proved the existence of extremal lattices for all star bodies of finite type:
\begin{prop} \label{existence of extremal lattice}
    Let $F:\R^k\rightarrow\R$ some non-negative continuous function, homogeneous of degree $n$. The function \begin{align*}
&\overline{F}:\SL_k(\R)/\SL_k(\Z)\rightarrow \R, \\&\overline{F}(\Lambda)=\inf_{\underline{x}\in\Lambda\backslash\underline{0}}\vert F(\underline{x})\vert\end{align*} is upper semi-continuous. As a consequence, extremal lattices exist.
\end{prop} 
\begin{proof}
Let $\Lambda_i\rightarrow\Lambda_0$ denote a converging sequence of $\SL_k(\R)/\SL_k(\Z).$ We want to show that $\overline{F}(\Lambda_0)\geq\limsup_i(\overline{F}(\Lambda_i)).$  By definition of $\overline{F}$, if $\Lambda_i$ is a sequence in $\SL_k(\R)/\SL_k(\Z)$ converging to the cusp, the shortest vector of $\Lambda_i$ converges to $0$ and by continuity of $F$, $$\overline{F}(\Lambda_i)\rightarrow0.$$ Hence, the function $\overline{F}$ is bounded and we can assume that $\overline{F}(\Lambda_i)$ is convergent. Denote this limit by $l$ and assume that $\overline{F}(\Lambda_0)< l$. Then we can find some vector $u\in\Z^k\backslash\underline{0},$ such that $\left\vert F\circ\Lambda_0(u)\right\vert<l$. Fix some $\epsilon$ smaller than $l-\left\vert F\circ\Lambda_0(u)\right\vert$. Then, for $i$ large enough we have $$\left\vert F\circ\Lambda_i(u)\right\vert<l-\epsilon$$ and hence $\overline{F}(\Lambda_i)<l-\epsilon,$ which is a contradiction. We have shown that $\overline{F}$ is bounded and semi-upper continuous and hence assumes its supremum.
\end{proof}

We thus see that the maximum of the spectrum of some form $F$ can be realized. A $K-$admissible lattice $\Lambda$ is called \textit{finitely minimized} if there exists $u\in\Lambda\backslash{0}$ such that $$\overline{F}(\Lambda)=F(u).$$ If the star body $K$ is bounded then every $K-$admissible lattice is trivially finitely minimized, but we can find unbounded star bodies of finite type without $K-$admissible extremal lattices. In fact, if the star body $K_F$ associated to the form $F$ is bounded then we can say significantly more:
\begin{prop}\label{Continuity}
Let $F$ denote an non-negative continuous function on $\R^k$, homogeneous of degree $n$, such that the star body $K_F$ associated to the form $F$ is bounded. Then the function \begin{align*}
&\overline{F}:\SL_k(\R)/\SL_k(\Z)\rightarrow \R, \\&\overline{F}(\Lambda)=\inf_{\underline{x}\in\Lambda\backslash\underline{0}}\vert F(\underline{x})\vert\end{align*} is continuous. Furthermore, if the origin is the only solution of the equation $$F(\underline{x})=0,$$ then $\overline{F}$ is everywhere positive.
\end{prop}
\begin{proof}
By definition, for any divergent sequence $\underline{x}_i\in\R^k$, $$\lim_iF(\underline{x}_i)\rightarrow\infty.$$ Therefore, there exists some compact ball $B$ in $\R^k$, that depends only on $F$, such that for every unimodular lattice $\Lambda$, $$\lim_{\underline{x}\in\Lambda\backslash\underline{0}}F(x,y)=\lim_{\underline{x}\in\Lambda\cap B\backslash\underline{0}}F(x,y).$$ Continuity and positivity follow immediately from compactness.
\end{proof}

As an immediate corollary, we obtain Mahler's Theorem:
\begin{proof}[Proof of Theorem \ref{Positive definite}:] By Lemma \ref{existence of extremal lattice} and Proposition \ref{Continuity}, $\Spec(P)=(a,M_P]$ for some $a\geq 0$ or $\Spec(P)=[b,M_P]$ for some $b>0.$ It is however clear, that by taking $\Lambda$ near the cusp of $\SL_2(\R)/\SL_2(\Z),$ $\inf_{\underline{x}\in\Lambda\backslash\underline{0}}\vert P(\underline{x})\vert$ becomes arbitrarily close to $0$ and hence: $$\Spec(P)=(0,M_P].$$    
\end{proof}

In many rigid examples of star bodies, the function $F$ possesses additional symmetries, such as linear actions. We call some $\gamma\in\SL_k(\R)$ an automorphism of $F$ if for every lattice $\underline{x}\in\R^k:$ $$F(\underline{x})=F(\gamma\cdot\underline{x}).$$ A star body is called \textit{automorphic} if it admits a group $\Gamma$ of automorphisms such that the following holds: There exists a bounded set $\overline{K}_{\Gamma}\subset\R^k$ depending only on $K$ and $\Gamma$ such that for every $\underline{x}\in K,$ there exists $\gamma\in\Gamma$ satisfying $$\gamma\cdot\underline{x}\in\overline{K}_{\Gamma}.$$Every bounded star body is trivially automorphic by taking $\Gamma=\{Id\}$ and $K=\overline{K}_{\Gamma}.$
\begin{thm}[Mahler]\label{Automorhpic}
Every automorphic star body possesses finitely minimized extremal lattices.    
\end{thm}
Using the invariance of the form $P(x,y)=xy$ by the diagonal action, we can deduce that the star body associated to the form $P$ is an automorphic star body and hence possesses finitely minimized extremal lattices. Of course, from Markoff's theorem we know much more about the behaviour of this automorphic star body near the top of the spectrum.

\section{Preliminaries on Continued Fractions}
\label{Preliminaries on Continued Fractions}
We list a number of well known facts regarding continued fractions, that could be found in any relevant textbook (see for example \cite{CasselsBook},\cite{CusickFlahive},\cite{KhinchinTextbook}).  \\ \\For any real number $x\notin\Q$ there exists a unique sequence of numbers $$[\alpha_0(x),\alpha_1(x),...,\alpha_N(x),...]\in \Z\times\mathbb{N}\times\mathbb{N}\times\mathbb{N}\times\dots,$$such that $$x=\alpha_0(x)+\frac{1}{\alpha_1(x)+\frac{1}{\alpha_2(x)+\frac{1}{\dots}}}.$$
We call this sequence the continued fractions sequence of $x$. For any $N\in\mathbb{N}$, the real number equal to$$[\alpha_N(x),\alpha_{N+1}(x),...]$$ will be denoted by $a_{N}(x)$. We call any finite sequence $[\alpha_0,\alpha_1,...,\alpha_N]\in \Z\times\mathbb{N}\times\mathbb{N}\times\dots\times\mathbb{N}$ a starting sequence. \\ \\ 
For rational numbers, this procedure terminates and we write $[\alpha_0(x),\alpha_1(x),...,\alpha_N(x),\infty]$ for $$\alpha_0(x)+\frac{1}{\alpha_1(x)+\frac{1}{\dots+\frac{1}{\alpha_N(x)}}}.$$ The continued fractions sequence for rational numbers is unique modulo the identification $$[\alpha_0(x),\alpha_1(x),...,\alpha_N(x),1,\infty]=[\alpha_0(x),\alpha_1(x),...,\alpha_N(x)+1,\infty].$$
For $x\in\R\setminus\Q$, the rational number $\frac{P_N(x)}{Q_N(x)}$ equal to $[\alpha_0(x),\alpha_1(x),...,\alpha_N(x),\infty]$
is called the $N^{\text{th}}$ convergent of $x$. The following Lemmata are standard: 

\begin{lemma}
    Let $x\in\R\setminus\Q.$ The following relation holds between the continued fraction sequence of $x$ and its convergents:
    $$Q_{N+1}(x)=\alpha_{N+1}(x)Q_N(x)+Q_{N-1}(x).$$
\end{lemma}

\begin{lemma} \label{Convergents Approximation}
    Let $x\in\R\setminus\Q.$ For any $N\in\mathbb{N},$ we have $$\left\vert x-\frac{P_N(x)}{Q_N(x)}\right\vert=\frac{1}{Q_N(x)(a_{N+1}(x)Q_N(x)+Q_{N-1}(x))}.$$
\end{lemma}
Hence, convergents provide a very good rational approximation. Conversely, the rational numbers that are good approximations of $x$ have to be convergents:

\begin{lemma} \label{Good approximation means convergent}
    Let $x\in\R\setminus\Q.$ If for some rational number $\frac{X}{Y}$ we have $$\left\vert x-\frac{X}{Y}\right\vert< \frac{1}{2Y^2},$$ then $(X,Y)=(P_N(x),Q_N(x))$ for some $N\in\mathbb{N}.$
\end{lemma}

Moreover, the convergents provide the best possible rational approximation in the following sense:

\begin{lemma} \label{Best rational approximation}
Let $x\in\R\setminus\Q.$ Then for all $Y<Q_N(x)$ and any $X$ we have $$\left\vert x-\frac{X}{Y}\right\vert>\left\vert x-\frac{P_N(x)}{Q_N(x)}\right\vert.$$
\end{lemma}

Regarding rational approximation of algebraic numbers, we have the following celebrated theorem:

\begin{thm}[Thue-Siegel-Roth]
Let $\alpha$ be some algebraric number and $\delta>0.$ There exists $N_0\in\mathbb{N}$ such that for all $X,Y\in \Z$ with $Y\geq N_0$: $$\left\vert\alpha-\frac{X}{Y}\right\vert>\frac{1}{Y^{2+\delta}}.$$
    
\end{thm}

\section{The star body problem for binary cubic forms}
\label{The star body problem for binary cubic forms}
We consider the action of $\GL_2\left(\mathbb{R}\right)$ on the space of binary cubic forms given by 
\begin{align}P\circ\begin{pmatrix}
    a & b \\ 
    c & d \\
\end{pmatrix}^{-1}(x,y)=P(ax+by,cx+dy).\notag\end{align} 
We define the action with the inverse so that if the roots of the equation $P(z,1)=\nolinebreak0$ are denoted by $\rho_1,\rho_2,\rho_3$, then the roots of the equation $P\circ T(z,1)=0$ are $T(\rho_1),T(\rho_2),T(\rho_3)$. The algebra of invariants for the action of $\SL_2(\R)$ on the space of binary cubic forms is a polynomial algebra in one variable generated by the discriminant polynomial of degree $4$ (\cite{InvariantTheory}):
\begin{thm}\label{Mordell Lemma} 
    If $P_1,P_2$ are two real binary cubic forms with the same discriminant, then there exists a linear transformation $T\in \SL_2\left(\mathbb{R}\right)$ such that $$P_1\circ T\left(x,y\right)=P_2\left(x,y\right).$$
\end{thm}

Since all binary cubic forms of negative (or positive) discriminant are equivalent under linear transformations, we have the following corollaries regarding the spectra of such forms. \\
\begin{corollary}
If $P_1,P_2$ are binary cubic forms of the same discriminant, then $$\Spec(P_1)=\Spec(P_2).$$
\end{corollary} \noindent Since the discriminant of a binary cubic form is a polynomial of degree 4, we have the following:

\begin{corollary}
The set $$\frac{\Spec(P)}{\sqrt[4]{\left\vert D_P\right\vert}}$$ depends only on the sign of the discriminant of the form $P$. 
\end{corollary}
\noindent Hence we can freely speak of the ``Spectrum of binary cubic forms with positive(resp. negative) discriminant". As a consequence, we have the following ``duality" phenomenon:
\begin{corollary}\label{Duality} If $P$ is some binary cubic form with discriminant $D_P$, $$\Spec(P)=\left\{\inf_{\Z^2\backslash\underline{0}}\left\vert G\left(x,y\right)\right\vert,G\,\text{varies over binary cubic forms of discriminant }D_P\right\}.$$

\end{corollary}

In other words, varying the form with a fixed lattice is the same as varying the lattice with a fixed form. In his original paper, Mordell introduced the form $$P(x,y)=x^3-xy^2-y^3$$ of $D_P=-23$, and showed that for any unimodular lattice $\Lambda$, there exists $v\in\Lambda\backslash\underline{0},$ such that $$\left\vert P(v)\right\vert\leq 1.$$
Since $$\inf\limits_{(x,y)\in \Z^2\backslash\underline{0}}\left\vert P(x,y)\right\vert=1,$$ he proved that $\Z^2$ is an extremal lattice for $P$ and the maximum point of  $\Spec(P)$ is equal to 1. Similarly, for positive discriminant he considered the form $$P(x,y)=x^3+xy^2-2x^2y-y^3$$ of $D_P=49$, and showed that for any unimodular lattice $\Lambda$, there exists $v\in\Lambda\backslash\underline{0},$ such that $$\left\vert P(v)\right\vert\leq 1.$$
Again, since $$\inf\limits_{(x,y)\in \Z^2\backslash\underline{0}}\left\vert P(x,y)\right\vert=1,$$ he proved that $\Z^2$ is an extremal lattice for $P$ and the maximum point of $\Spec(P)$ is equal to $1$.  

\subsection{The case of negative discriminant} \label{The case of negative discriminant} \leavevmode\newline \leavevmode\newline
Motivated by Mordell's findings (\cite{Mordell}) and Corollary \ref{Duality},  we introduce the following continuous family of forms. For $t\geq 0$ set:
$$P_{t}(x,y)=\left(x-\rho y\right)\left((x+\frac{\rho}{2}y)^2+(\frac{3}{4}\rho^2-1)(1+t^2)y^2\right),$$ where $\rho$ denotes the real root of $x^3-x-1=0.$ For example, we have $P_0(x,y)=x^3-xy^2-y^3,$ the extremal form given by Mordell. The proof of Theorem \ref{D<0} is very simple:
\begin{proof}[Proof of Theorem \ref{D<0}]
We start by noticing that $\left\vert P_t(x,y)\right\vert\geq \left\vert P_0(x,y)\right\vert.$
Hence, we have that $$\inf_{(x,y)\in \Z^2\backslash\underline{0}}\left\vert P_t(x,y)\right\vert\geq\inf_{(x,y)\in \Z^2\backslash\underline{0}}\left\vert P_0(x,y)\right\vert=1.$$ On the other hand, $P_t(1,0)=1$ and thus $$\inf_{(x,y)\in \Z^2\backslash\underline{0}}\left\vert P_t(x,y)\right\vert=1.$$ Since $\lim\limits_{t\rightarrow\infty}D_{P_t}\rightarrow\infty,$ we obtain the Theorem. \\ \\ 
\end{proof}

\subsection{The case of positive discriminant} \label{The case of positive discriminant} \leavevmode\newline \leavevmode\newline
Unlike the case of $D<0,$ for $D>0$ we cannot take advantage of the positive definite part of the form to trivially construct a one parameter family of unimodular transformations, with infima covering the entire spectral interval. In fact, as we show in Proposition \ref{Non-existence of continuous curve}, it is not possible to find a continuous family of lattices, connecting $0$ to the maximum point of the spectrum continuously. Similarly to the definition of finitely minimized lattices, we have:  
\begin{defn}\label{fntly minimized}

We call a form $P$ finitely minimized if the value $$m(P)=\inf_{(x,y)\in{\Z^2}\backslash\underline{0}}\vert P\left(x,y\right)\vert$$ is a minimum (attained by the function $P$ for some $\left(x,y\right)\neq (0,0))$. 

\end{defn}

It is obvious that if $P(z,1)=0$ has a rational root, then $m(P)=0$. However, $m(P)$ is almost everywhere positive and, in fact, almost all binary cubic forms $P$ of $D_P>0$ are finitely minimized:

\begin{prop}\label{Finitely Minimized}

For almost all binary cubic forms $P$ of positive discriminant, $P$ is finitely minimized.
\end{prop}

\begin{proof}

For any positive number $l$, define the set $$S_l=\left\{\theta \in \mathbb{R}\text{, such that for all }x\in\Z,y\in \mathbb{N}: \left\vert\theta-\frac{x}{y}\right\vert>\frac{1}{ly^{\frac{5}{2}}} \right\}.$$By Khinchin's Theorem, $\bigcup_{l\in \mathbb{N}}S_l$ has full measure. Pick some $\rho_1,\rho_2,\rho_3\in S_l$ pair-wise distinct and some $c\in\R$ with $c\neq0$. Consider the form $$P(x,y)=c(x-\rho_1y)(x-\rho_2y)(x-\rho_3y),$$ and assume that $P$ is not finitely minimized. By definition, there exists a sequence of $(x_i,y_i)\rightarrow\infty$ in $\Z^2$ and some absolute constant $M$, not depending on $l$ such that $$\left\vert P\left(x_i,y_i\right)\right\vert=\left\vert y_i^3\left(\frac{x_i}{y_i}-\rho_1\right)\left(\frac{x_i}{y_i}-\rho_2\right)\left(\frac{x_i}{y_i}-\rho_3\right)\right
\vert\leq M. $$ It is not hard to see that $\frac{x_i}{y_i}$ stays in some compact set for all $i$. By possibly passing to a subsequence, assume that $\frac{x_i}{y_i}$ is convergent. Notice that the limit of this sequence must be one of the $\rho_j$'s, say $\rho_1$ without loss of generality. We have: 
$$M\geq\left\vert P(x_i,y_i)\right\vert= \left\vert (1+o_i(1))(\rho_1-\rho_2)(\rho_1-\rho_3)y_i^3\left(\frac{x_i}{y_i}-\rho_1\right)\right\vert\gg\sqrt{\left\vert y_i\right\vert}\rightarrow\infty,$$ by definition of $S_l$, which is a contradiction.

\end{proof} 
By the Thue-Siegel-Roth theorem, we know that every algebraic number is in some $S_l,$ and hence we see that even though binary cubic forms do not admit linear actions in the sense of \ref{Automorhpic}, by Mordell's findings, we know that: 
\begin{corollary}\label{Finitely minimized extremal cubic}
Every extremal lattice associated to the star body of a binary cubic form of non-zero discriminant is finitely minimized.    
\end{corollary}
Regarding the continuity of the function $\Lambda\rightarrow m(P\circ\Lambda)$ we have:
\begin{prop} \label{Non-existence of continuous curve}
    Let $P(x,y)$ be a binary form of degree $n$ with non-vanishing discriminant and assume that $P(z,1)=0$ has only real roots. Let $u:[0,1]\rightarrow \SL_2(\R)/\SL_2(\Z)$ denote a continuous map on the space of unimodular lattices and let  $$U:[0,1]\rightarrow \R_{\geq0}, \,\,U(t)= \inf_{(x,y)\in\Z^2\backslash\underline{0}}\vert P\circ u(t)(x,y)\vert.$$ Then, $$U\text{ is continuous}\iff u\text{ is constant.}$$
\end{prop}
\begin{proof}
    The reverse direction is clear. Assume the function $U$ is continuous and that $u$ is non-constant. Suppose there exists $t_0\in[0,1]$ such that $U(t_0)>0.$ Without loss of generality we can assume that $u$ is not constant in any open set containing $t_0$. By continuity of $u$, we can choose continuous functions $u_1,u_2,...,u_n:[0,1]\rightarrow \R,$ so that for every $t\in[0,1],$ $\{u_1(t),u_2(t),...,u_n(t)\}$ is the set of roots of $P\circ u(t).$ Then, there exists some $i_0$, such that $u_{i_0}$ is not constant in any open set containing $t_0.$ Therefore, a sequence $t_i\rightarrow t_0$ exists so that $u_{i_0}(t_i)\in\Q.$ It then follows that $U(t_i)=0$ and thus $U(t_0)=0.$ 
\end{proof}

Therefore, as we will see, the interval nature of the spectral set in question is  not a consequence of continuity, but rather an outcome of the chaos in the distribution of continued fractions. The proof of Theorem \ref{D>0} will be concluded in Section \ref{Fixed point perturbations and intervals in the spectrum}. For now, we prove a weaker version of this theorem, which we will later, in Theorem \ref{Density theorem}, generalize to forms of degree $n\geq 4.$

\begin{thm}
    Let $P$ denote a binary cubic form of positive discriminant. Then $$\overline{\Spec(P)}=\left[0,\sqrt[4]{\frac{D_P}{49}}\right].$$
\end{thm}

\begin{proof}

Let $\rho>\chi>\psi$ denote the real roots of $x^3+x^2-2x-1=0$ and recall that $P(x,y)=x^3+x^2y-2xy^2-y^3$ is the extremal form for the lattice $\Z^2$, as proved by Mordell.  Let the continued fraction sequence of $\rho$ be denoted as $\left[\alpha_0(\rho),\alpha_1(\rho),\alpha_2(\rho),...\right]$ and $\frac{P_N(\rho)}{Q_N(\rho)}$ will be the $N^{\text{th}}$ convergent of $\rho.$ Fix some real $c\geq1 $ and  for any $N\in \mathbb{N}$ define $$\rho(c,N)=\left[\alpha_0(\rho),\alpha_1(\rho),\alpha_2(\rho),...,\alpha_{N}(\rho),\lfloor c(\rho-\chi)(\rho-\psi)Q_N(\rho)\rfloor,1,1,1,...\right].$$ In other, words we construct a real number by defining the first part of the continued fraction sequence to be the same as that of $\rho,$ picking $\alpha_{N+1}$ to be a variable, and then fixing the tail to resemble a badly approximatable number. We now study $m(P_{c,N})$ for $$P_{c,N}(x,y)=(x-\rho(c,N)y)(x-\psi y)(x-\chi y).$$ Since $\rho$ is an algebraic number, we have by the Thue-Siegel-Roth Theorem, that there exists some integer $N_0,$ such that for all $i\geq N_0$: \begin{align}
\alpha_{i+1}(\rho)&\leq \sqrt{Q_{i}(\rho)}.  \label{Diophantine cutting} \end{align} By Lemma \ref{Convergents Approximation},  we know that \begin{align}\left\vert\rho-\frac{P_i(\rho)}{Q_i(\rho)}\right\vert=\frac{1}{Q_{i}(\rho)(a_{i+1}(\rho)Q_i(\rho)+Q_{i-1}(\rho))}. \notag\end{align} Therefore, by Eq.\ref{Diophantine cutting} and Lemma \ref{Good approximation means convergent}, there exists some absolute constant $M$, so that for all $Y>Q_{N_0}$: $$Y^3\left\vert\rho-\frac{X}{Y}\right\vert\geq M\sqrt{Y}.$$ Now, since $\alpha_i(\rho)=\alpha_i(\rho(c,N))$ for $i\leq N-1$, we have an analogous statement for $\rho(c,N)$ by possibly increasing the value of the constant $M$. That is, for $N>N_0$ and for all $c\geq1$:  $$Q_{N}(\rho(c,N))>Y>Q_{N_0}(\rho(c,N))\implies
Y^3\left\vert\rho(c,N)-\frac{X}{Y}\right\vert\geq M\sqrt{Y}.$$ 
Now, since for $j\geq1$, $\alpha_{N+j+1}(\rho(c,N))=1,$ we know again from Lemma \ref{Convergents Approximation} that $$Y> Q_{N}(\rho(c,N))\implies Y^3\left\vert\rho(c,N)-\frac{X}{Y}\right\vert\geq MY.$$

We therefore deduce that by picking $N_0$ large enough, $P_{c,N}$ is finitely minimized by either $Y\leq Q_{N_0}(\rho(c,N))$ or by $Y=Q_{N}(\rho(c,N)).$ Since Eq. \ref{Diophantine cutting}, holds for $\chi$ and $\psi$ as well we deduce that there exists some compact region  $\Omega=[0,T]^2$ (with $T$ an absolute constant), such that 
$$m(P(c,N))=\min\left(\min_{(x,y)\in\Omega\backslash\underline{0}}\left\vert P_{c,N}(x,y)\right\vert,\left\vert P_{c,N}(P_N(\rho(c,N)),Q_N(\rho(c,N)))\right\vert\right).$$

However, we have
$$\min_{(x,y)\in\Omega\backslash\underline{0}}\left\vert P_{c,N}(x,y)\right\vert=\left(1+o_N(1)\right)\min_{(x,y)\in\Omega\backslash\underline{0}}\left\vert P(x,y)\right\vert=1+o_N(1),$$ and

\begin{align*}&P_{c,N}(P_N(\rho(c,N)),Q_N(\rho(c,N)))=(1+o_N(1))(\rho-\chi)(\rho-\psi)\frac{Q_N(\rho(c,N))}{\alpha_{N+1}(\rho(c,N))}=\notag \\&(1+o_N(1))(\rho-\chi)(\rho-\psi)\frac{Q_N(\rho)}{\lfloor c(\rho-\chi)(\rho-\psi)Q_N(\rho)\rfloor}=\frac{1}{c}(1+o_N(1)).\notag\end{align*}

Taking everything together, we have that for $c\geq 1$ $$m(P_{c,N})=\frac{1}{c}(1+o_N(1)).$$

Since the discriminant of $P_{c,N}$ is equal to $49(1+o_N(1)),$ we derive the Theorem.

\end{proof}
\begin{remark} \label{Remark 3}
    Even though we have used the Theorem of Thue-Siegel-Roth to obtain Theorem \ref{D>0}, the Theorem of Thue would suffice as for our purposes any saving on the exponent $3$ would suffice. In fact, we will see later, that even the use of that can be avoided. In that sense, our proof is completely elementary.
\end{remark}

If one is a bit more careful with the crude diophantine cutting of the form $``1,1,1,..."$ that we performed, it is easy to see that we can construct a set of infima that is not simply dense in the spectral set, but also of full measure. We do not pursue this result here, even though it does not require much effort, since we will later give the corresponding ``full measure" result for any binary form of degree $n\geq 3$ and non-zero discriminant. Furthermore, note that we have achieved full density of the spectrum in an aribatry neighborhood of the extremal lattice; that is the lattices that we used to fill the spectral interval where chosen arbitrarily close to the identity. \\

We also note that the family of binary cubic forms we constructed in the proof of Theorem \ref{D>0}, have the property that their largest root lies inside $\Q(\sqrt{5}).$ Of course, we could have chosen any number with Diophantine exponent equal to 2 to perform the $``1,1,1,...,1,..."$ cutting for all the roots. We have proved the following: 
\begin{thm}
For any binary cubic form $P$ and any real quadratic field $\Q({\sqrt{d}})$, define $$\Spec(P,d)=\left\{\inf_{\underline{x} \in\Lambda \setminus \underline{0}}\left\vert P(\underline{x})\right\vert, \Lambda \subset \R^2 \text{ a unimodular lattice such that $P\circ\Lambda$ has roots in $\Q(\sqrt{d})$} \right\}.$$ Then $$\overline{\Spec(P,d)}=\begin{cases}
    \left[0,\sqrt[4]{\frac{D_P}{49}}\right] &\text{for $D_P>0,$ }\\[10pt]
    \left[0,\sqrt[4]{\frac{-D_P}{23}}\right] &\text{for $D_P<0.$ }\\
\end{cases}.$$
\end{thm}
We stated the theorem for any binary cubic form, since the proof of Theorem \ref{D>0} could have also been applied for the case of $D<0$. \\

\section{The sets $B_{\epsilon}$ and $E^{\eta}$}
\label{The sets}

We now turn our focus onto the case of binary forms of higher degree. The situation for forms of degree $n\geq4$ differs to the cases already discussed, since we have no analogue of Lemma \ref{Mordell Lemma}. For every $n\in\mathbb{N}$ the ring of invariants under the action of $\SL_2$ is a finitely generated graded algebra over the base field of the complex numbers, as proved by Gordan (\cite{InvariantTheory}). However, for $n\geq4$ this ring is not generated by the discriminant alone and hence there are uncountably many $\SL_2(\R)-$orbits. \\

In proving Theorem \ref{D>0}, we used our knowledge of the extremal lattice and specifically the fact that the roots of the extremal form are algebraic. Even though we used the theorem of Thue-Siegel-Roth, we are not essentially using the fact that the extremal form is algebraic, but simply that its roots have Diophantine exponent less than $3$.

\begin{defn}
    Let $x\in\R$. The Diophantine exponent of $x$ is  
    $$E(x)=\sup\{r:\exists \,C>0 \,\, \text{s.t.}\,\,\left\vert x-\frac{p}{q}\right\vert\leq\frac{C}{q^r}\text{ for infinitely many}\, (p,q)\in\Z\times\mathbb{N}\}.$$ We will also be denoting $E_s$ the set of real numbers with Diophantine exponent $s$. Since for $n\geq4,$ almost all $\SL_2(\R)-$orbits do not contain rational representatives, forms with roots in $E_2$ will be the analogue of the extremal forms with good Diophantine properties discovered by Mordell. 
\end{defn} We can express $E(x)$ as a function of the continued fraction sequence of $\alpha$: 
$$E(x)=2+\limsup\frac{\log\alpha_N(x)}{\log Q_N(x)}.$$
We say \begin{itemize}
    \item $x$ is sub-critical if $E(x)<n$,
    \item $x$ is critical if $E(x)=n$,
    \item $x$ is super-critical if $E(x)>n$.
\end{itemize} 
Of course whether some $x\in\R$ is sub-critical, critical or super-critical depends on the degree $n$ we are discussing. However, this will be clear from the context and hence we omit the dependence. If $P$ is some binary form with a super-critical distinct real root $\rho$, then 

$$\inf_{(x,y)\in\Z\times\mathbb{N}}\left\vert P(x,y)\right\vert\ll\limsup \left\vert y^n(\frac{x}{y}-\rho)\right\vert=0. $$

If we wanted to summarize in a line, why the spectrum of binary $n$-forms for $n\geq3$ is so different from the Markoff Spectrum, it is because for $n=2$, there are no sub-critical real numbers.\\ \\
We need some more definitions regarding Diophantine properties of real numbers:
\begin{defn}
Let $\rho\in\R$ and $\eta>0.$ We define $E^{\eta}(\rho)$ to be the set of points $x$ such that $$\left\vert\frac{X}{Y}-x\right\vert<\frac{1}{Y^{2+\eta}}\implies \left\vert\frac{X}{Y}-\rho\right\vert<\frac{2}{Y^{2+\eta}}.$$
\end{defn}
The definition of $E^{\eta}(\rho)$ is a quantitative way to describe the real points with sequence of convergents almost containing the sequence of convergents of $\rho$. \\ \\We make the following conventions for the rest of the paper:\\ \\
\textit{Notational convention:} From now on, we assume that all discussed binary forms are homogeneous of degree $n\geq 3$ and have non vanishing discriminant. By ``roots of $P$" we will be referring to the roots of the equation $P(z,1)=0.$ These will be denoted by $\rho_i$, $i\in\{1,2,...,n\}$, where the first $k$ of them are real and the rest are complex coming in conjugate pairs. Since the case of $k=0$ has already been discussed, we assume that there exists at least one real root. We denote the largest one by $\rho_1$ and we can assume that it is positive, since we can always conjugate with some lattice in $\SL_2(\mathbb{Z}).$
\begin{defn}
Let $P$ denote some binary form of degree $n\geq 3$ with $k$ distinct real roots $\rho_{i},$ $i=1,2,...,k$ and denote by $\Lambda_0$ an extremal lattice. For some $\epsilon>0$, we call a lattice $\Lambda\subset\R^2$ $\epsilon-$almost extremal if 
    \begin{itemize}
        \item $\vert\vert \Lambda-\Lambda_0\vert\vert<\epsilon$ and \item$$\inf_{(X,Y)\in\Z\times\mathbb{N}}\left\vert Y^n(\frac{X}{Y}-\rho_{\Lambda,i})\right\vert>\inf_{(X,Y)\in\Z\times\mathbb{N}}\left\vert Y^n(\frac{X}{Y}-\rho_i)\right\vert-\epsilon,$$ for $i\in\{1,2,...,k\},$ where $\rho_{\Lambda,i}$ are the real roots of $P\circ\Lambda.$
    \end{itemize} 
We denote the set of $\epsilon-$almost extremal lattices by $\text{AEL}(\epsilon).$
\end{defn}
Notice that the absence of a spectral gap is equivalent to the statement that for every $\epsilon,$ AEL($\epsilon$) contains non-extremal lattices.
\begin{defn}
Let $P$ denote some binary form of degree $n\geq 3$ with $k$ distinct real roots $\rho_{i},$ $i=1,2,...,k$ and $X\subset\R^k.$ For any $\eta,\epsilon>0,$ we define AEL$(\epsilon,X,\eta)$ the set of those $T\in\SL_2(\R)$ such that:  
    \begin{itemize}
    \item $T\in\text{AEL}(\epsilon),$ where we identify $T$ with its corresponding lattice,
        \item $\left(T(\rho_1),T(\rho_2),...,T(\rho_k)\right)\in X$,
        \item $T(\rho_i)\in E^{\eta}(\rho_i).$
    \end{itemize} 
\end{defn}

The proof of the following theorem will be completed in Section \ref{The structural theorem and density of approximant lattices}.  
\begin{thm}\label{Flexibility of the extremal lattice}
Let $P$ denote some binary form of degree $n\geq 3$ with $k$ distinct real roots $\rho_{i},$ $i=1,2,...,k$. Let  $X_i\subset\R$ and set  $$a_i=\underline{\lim}_{\epsilon\rightarrow0}\frac{\lambda(B(\rho_i,\epsilon)\bigcap X_i)}{\lambda(B(\rho_i,\epsilon))}.$$ For every $\eta,\epsilon>0$, we have $$\overline{\lim}_{\epsilon_2\rightarrow0}\frac{\mu(\text{AEL}(\epsilon,X_1\times X_2\times...\times X_K,\eta))}{\mu(B(id_2,\epsilon))}\geq 1-k(1-\min_i{a_i}).$$
Here, we denote the Euclidean ball of radius $\epsilon$ around $\rho$ by $B(\rho,\epsilon)$ and the Haar measure on $\SL_2(\R)$ by $\mu.$
\end{thm}
For our purposes, we will use Theorem \ref{Flexibility of the extremal lattice} only for $X_i\subset\R$ conull sets. We see, by Theorem \ref{Flexibility of the extremal lattice}, that, unlike the case of indefinite quadratic forms, for forms of higher degree, we have a lot of flexibility near our extremal lattices. As an immediate corollary, we obtain: 
\begin{corollary}
    For any binary form $P$ of degree $n\geq3$ and non-zero discriminant, $\Spec(P)$ does not have an isolated maximum.
\end{corollary}

We devote the rest of this section to building up machinery for the proof of Theorem \ref{Flexibility of the extremal lattice}. 

\begin{defn}
    For any $\rho\in\R,$ we set $$m(\rho)=\inf_{(X,Y)\in\Z\times\mathbb{N}}\left\vert Y^n(\frac{X}{Y}-\rho)\right\vert.$$ We say $\rho_1$ is $\epsilon-$Diophantine larger than $\rho_2$, if $$m(\rho_1)>(1-\epsilon)m(\rho_2).$$
    Denote by $B_{\epsilon}(\rho)$ the real numbers $\epsilon-$Diophantine larger than $\rho$. 
\end{defn}

The following Lemma is the reason we are interested in the distribution of $B_{\epsilon}(\rho)$.

\begin{lemma} \label{B relevance}
Let $P_1,P_2$ be binary forms of degree $n$ of the same discriminant with roots denoted by $\rho_{1,i},\rho_{2,i}$ for $i\in\{1,2,...,n\}$. Fix $\epsilon>0.$ If \begin{enumerate}
    \item $$\rho_{1,i}\in B_{\epsilon}(\rho_{2,i})\bigcap B(\rho_{2,i},\epsilon),$$ for $i\in\{1,2,...,k\}$
    \item $$\rho_{1,i}\in B(\rho_{2,i},\epsilon),$$ for $i\in\{k,k+1,...,n\},$
\end{enumerate}  then $$m(P_1)\geq m(P_2)-\epsilon\cdot O_{P_1,P_2}(1),$$ where $O_{P_1,P_2}(1)$ depends only on the size of the coefficients of $P_1$ and $P_2$.
\end{lemma}

\begin{proof}
Notice that for a form $P$ of non-zero discriminant, we have similarly to the proof of Theorem \ref{D>0}, that  $$m(P)=\min\left({\inf_{(X,Y)\in[-T,T]\backslash\underline{0}}\left\vert P(X,Y)\right\vert,\min_i(m(\rho_i)\prod_{j\neq i}\left\vert\rho_i-\rho_j\right\vert)}\right)+o_P(\frac{1}{T}),$$ as $T\rightarrow\infty.$ The Lemma easily follows from this observation.
\end{proof}

\begin{corollary}
 Let $P$ be binary forms of degree $n$ with roots denoted by $\rho_{i},$ $i\in\{1,2,...,n\}$ and let $\Lambda_0$ denote an extremal lattice. For every $\epsilon>0$, there exists $\delta>0$ such that if for some unimodular lattice $\Lambda\subset\R^2$, the roots $\rho_{\Lambda,i},$ $i\in\{1,2,...,n\}$ of $P\circ\Lambda$ satisfy  \begin{enumerate}
    \item $$\rho_{\Lambda,i}\in B_{\delta}(\rho_{i})\bigcap B(\rho_{i},\delta),$$ for $i\in\{1,2,...,k\}$
    \item $$\rho_{\Lambda,i}\in B(\rho_{i},\delta),$$ for $i\in\{k,k+1,...,n\},$
\end{enumerate}  then $$\Lambda\in\text{AEL}(\epsilon).$$ 
\end{corollary}

For the proof of Theorem \ref{Flexibility of the extremal lattice}, we will need to show that the set $$B_{\epsilon}(\rho)\cap E^{\eta}(\rho)\cap B(\rho,\epsilon)$$ has many points. Before we demonstrate how to obtain such a result, we need two lemmata regarding the distribution of continued fractions.

\begin{lemma} \label{Continued Fraction Distribution}
    Let $[\alpha_0,\alpha_1,...,\alpha_N]$ denote some starting sequence. Then there exists positive absolute constants $A,B$ such that for any $k\in\mathbb{N}$,$$\frac{A}{k^2Q_N^2}<\lambda\left(x\in\R:\,\,\alpha_{i}(x)=\alpha_i \text{ for all $i\leq N$ and }\alpha_{N+1}(x)=k\right)<\frac{B}{k^2Q_N^2},$$ where $Q_N$ is the denominator of the rational number with continued fraction $[\alpha_0,\alpha_1,...,\alpha_N]$ in simple form.
\end{lemma}

\begin{proof}
    Define the two rational numbers 
    \begin{align}
        &  z_1=[\alpha_0,\alpha_1,...,\alpha_N,k+1,\infty] \notag \\   &z_2=[\alpha_0,\alpha_1,...,\alpha_N,k,\infty], \notag    
    \end{align}

    and notice that $$\lambda\left(x\in\R:\,\,\alpha_{i}(x)=\alpha_i(r) \text{ for all $i\leq N$ and }\alpha_{N+1}(x)=k\right)=\left\vert z_1-z_2\right\vert.$$

    By Lemma \ref{Convergents Approximation}, we know that $\left\vert z_1-z_2\right\vert=\frac{1}{(kQ_N+Q_{N-1})((k+1)Q_N+Q_{N-1})},$ from which the result follows.

\end{proof}

Hence by observing that for some fixed number $\rho$ and some large natural number $N\in\mathbb{N},$ $\alpha_{N}(\rho)$ needs to be of size $\sim Q_{N-1}^{n-2}$ to ``affect" $m(\rho),$ we see that for most numbers, $m(\rho)$ is determined by the smaller integer pairs $(X,Y)$ rather than the larger ones. We quantify this by showing:

\begin{lemma} \label{Measure of Cutting}
Let $[\alpha_0,\alpha_1,...,\alpha_N]$ denote some starting sequence and let $\eta>0$. Then there exists some absolute constant $C>0$, such that  $$\frac{\lambda\left(\left\{x\in\R:\,\,\alpha_{i}(x)=\alpha_i \text{ for all $i\leq N$ and } \alpha_{N+i+1}(x)<{Q_{N+i}^{\eta}(x)}\text{ for all $i\geq0$}\right\}\right)}{\lambda\left(\left\{x\in\R:\,\,\alpha_{i}(x)=\alpha_i \text{ for all $i\leq N$ } \right\}\right)}>1-
C\phi^{-\eta N},$$ where $\phi$ denotes the golden ratio.
    
\end{lemma}

\begin{proof}
By Lemma \ref{Continued Fraction Distribution}, we know that $$\lambda\left(\left\{x\in\R:\,\,\alpha_{i}(x)=\alpha_i \text{ for all $i\leq N$ } \right\}\right)\sim \frac{1}{Q_N^2}.$$ If for some $x$ with starting sequence that of the statement of the Lemma, there exists $i\geq0$ such that $\alpha_{N+i+1}(x)\geq {Q_{N+i}^{\eta}(x)},$ then by Lemma \ref{Convergents Approximation} we can find a rational number $\frac{X}{Y}$ so that $Y>Q_N$ and $$\left\vert x-\frac{X}{Y}\right\vert<\frac{1}{Y^{2+\eta}}.$$ Furthermore, since $x$ has $\frac{P_N}{Q_N}$ as a convergent, we can assume $C_1Y<X<C_2Y,$ where $C_2-C_1=O(\frac{1}{Q_N^2}).$ We obtain the bound  
\begin{align}
    &\frac{\lambda\left(\left\{x\in\R:\,\,\alpha_{i}(x)=\alpha_i \text{ for all $i\leq N$ and } \alpha_{N+i+1}(x)<{Q_{N+i}^{\eta}(x)}\text{ for all $i\geq0$}\right\}\right)}{\lambda\left(\left\{x\in\R:\,\,\alpha_{i}(x)=\alpha_i \text{ for all $i\leq N$ } \right\}\right)}\geq \notag \\& 1-O(Q_N^2\sum_{\substack{C_1Y<X<C_2Y,\\Y>Q_N}}\frac{1}{Y^{{2+\eta}}})\geq 1-O(Q_N^2\sum_{Y>Q_N}\frac{Y}{Q_N^2}\frac{1}{Y^{2+\eta}})\geq\notag\\& 1-O(\sum_{Y>Q_N}\frac{1}{Y^{1+\eta}})\geq 1-O(\phi^{-\eta N}).\notag 
\end{align}
Here we used the fact that $Q_N\geq F_N$, where $F_N$ stands for the $N^{\text{th}}$ Fibonacci number. 
\end{proof} The following technical lemma is the key to constructing points inside $$B_{\epsilon}(\rho)\cap E^{\eta}(\rho).$$
\begin{lemma} \label{S-Lemma}

Let $n\geq3$ a natural number, $\rho\in\R$, $\epsilon>0$ and $\eta\in (0,n-2)$. For any $N\in\mathbb{N}$ and any positive integer $h\leq\alpha_N(\rho)$ define $S_{\rho,\epsilon,N,h,\eta}$ to be the set of real numbers $x$ with 

$$\alpha_i(x)\begin{cases}
=\alpha_i(\rho) &\text{for $i<N$ },\\\in \left[h,\min\left((1-\epsilon)^{-1}\left(Q^{n-2}_{N-1}(\rho)m^{-1}(\rho)+1\right)-1,(1+\epsilon)a_N(\rho)\right)\right] & i=N, \\<{Q_{i-1}^{\eta}(x)}&\text{for all $i>N.$}\end{cases}$$ 
Then for every $\epsilon>0,$ there exists $N_0$ such that for every $N\geq N_0$, every $h\leq a_N(\rho)$ and every $\eta>0$:
$$S_{\rho,\epsilon,N,h,\eta}\subset B_{\epsilon}(\rho)\cap E^{\eta}(\rho).$$

\end{lemma}
\begin{proof}
If $m(\rho)=0$ then $B_{\epsilon}(\rho)=\R$. We assume thus that $m(\rho)>0.$ The fact that $S_{\rho,\epsilon,N,h,\eta}\subset E^{\eta}(\rho)$ follows immediately from the definition of the set $S_{\rho,\epsilon,N,h,\eta}$. As for $B_{\epsilon}(\rho),$ let $x\in S_{\rho,\epsilon,N,h,\eta}.$ Notice that for $Y>Q_{N-1}(x),$ we know by Lemmata \ref{Convergents Approximation} and \ref{Good approximation means convergent} and the fact that $a_i(x)<{Q_{i-1}^{\eta}(x)}$ for $i>N,$ that there exists some absolute constant $C>0$, such that $$\left\vert Y^nx-XY^{n-1}\right\vert\geq CY^{n-2-\eta}.$$ Hence, there exists $N_0$ such that for $N\geq N_0$: $$m(x)=\min_{0<Y<Q_N(x),\vert X\vert<P_N(x)}\left\vert Y^nx-XY^{n-1}\right\vert.$$
By Lemma \ref{Good approximation means convergent}, we know that there exists some constant $T$ that depends only on $\rho$ such that $$m(x)=\min\left(\min_{Y<T,\vert X\vert<T}\left\vert Y^nx-XY^{n-1}\right\vert,\min_{i\leq N-1}\left\vert Q_i^n(x)x-P_i(x)Q_i^{n-1}(x)\right\vert\right).$$
Now since the constant $T$ is fixed and $\bigcap_{N}S_{\rho,\epsilon,N,h,\eta}=\{\rho\}$, we have that by picking $N_0$ large enough: $$\min_{0<Y<T,\vert X\vert<T}\left\vert Y^nx-XY^{n-1}\right\vert\geq m(\rho)(1-\epsilon).$$ Hence, we are done in the case of $$\min_{0<Y<T,\vert X\vert<T}\left\vert Y^nx-XY^{n-1}\right\vert\leq\min_{i\leq N-1}\left\vert Q_i^n(x)x-P_i(x)Q_i^{n-1}(x)\right\vert.$$
Assume now that $$\min_{i\leq N-1}\left\vert Q_i^n(x)x-P_i(x)Q_i^{n-1}(x)\right\vert<\min_{0<Y<T,\vert X\vert<T}\left\vert Y^nx-XY^{n-1}\right\vert.$$ For $i=N-1$ we have by Lemma \ref{Convergents Approximation}, that if $N$ is sufficiently large: $$\left\vert Q_{N-1}^n(x)x-P_{N-1}(x)Q_{N-1}^{n-1}(x)\right\vert\geq\frac{Q_{N-1}^{n-2}(x)}{\alpha_{N}(x)+1}\geq\frac{Q_{N-1}^{n-2}(x)}{(1-\epsilon)^{-1}Q_{N-1}^{n-2}(x)m^{-1}(\rho)}=(1-\epsilon)m(\rho),$$
where here we used that $\alpha_{N}(x)+1\leq(1-\epsilon)^{-1}Q_{N-1}^{n-2}(x)m^{-1}(\rho)$. We are left with showing $$\min_{i\leq N-2}\left\vert Q_i^n(x)x-P_i(x)Q_i(x)^{n-1}\right\vert>(1-\epsilon)m(\rho).$$ Assume $$\min_{i\leq N-1}\left\vert Q_i^n(x)x-P_i(x)Q_i^{n-1}(x)\right\vert=\left\vert Q_{k}^n(x)x-P_k(x)Q_k^{n-1}(x)\right\vert,$$ for some $T<Q_k\leq Q_{N-1}.$ Then we can deduce that for some absolute constant $C$ that depends only on $\rho$, $$\alpha_{k+1}(x)>CQ^{n-2}_k(x),$$ and hence by the triangle inequality,

\begin{multline}\left\vert Q_{k}^n(x)x-P_k(x)Q_k^{n-1}(x)\right\vert\geq  \left\vert Q_{k}^n(\rho)\rho-P_k(\rho)Q_k^{n-1}(\rho)\right\vert\\-\frac{Q_k^{n-1}(\rho)\left\vert a_{k+1}(\rho)-a_{k+1}(x)\right\vert}{(Q_k(\rho)a_{k+1}(\rho)+Q_{k-1}(\rho))(Q_k(\rho)a_{k+1}(\rho)+Q_{k-1}(\rho))}.\notag
\end{multline}

Since $\alpha_{k+1}(\rho)=\alpha_{k+1}(x),$ we have $\left\vert a_{k+1}(\rho)-a_{k+1}(x)\right\vert<1.$
We deduce that$$\frac{Q_k^{n-1}(\rho)\left\vert a_{k+1}(\rho)-a_{k+1}(x)\right\vert}{(Q_k(\rho)a_{k+1}(\rho)+Q_{k-1}(\rho))(Q_k(\rho)a_{k+1}(\rho)+Q_{k-1}(\rho))}<\frac{1}{C^2Q_k^{n-1}}<\frac{1}{C^2T^{n-1}}.$$ By possibly increasing the value of $T$ (only in terms of the fixed $\epsilon$), we hence get $$\left\vert Q_{k}^n(x)x-P_k(x)Q_k^{n-1}(x)\right\vert\geq \left\vert Q_{k}^n(\rho)\rho-P_k(\rho)Q_k^{n-1}(\rho)\right\vert-\epsilon m(\rho)\geq (1-\epsilon)m(\rho).$$

\end{proof}

\section{The structural theorem and density of approximant lattices} \label{The structural theorem and density of approximant lattices}
By Lemmata \ref{Measure of Cutting} and \ref{S-Lemma}, we have a way of constructing points in $B_{\epsilon}(\rho)\cap E^{\eta}(\rho)$ close to $\rho$. For the proof of Theorem \ref{Flexibility of the extremal lattice}, we would like to show that $B_{\epsilon}(\rho)\cap E^{\eta}(\rho)$ has density arbitrarily close to $1$ as we restrict to smaller and smaller open sets around $\rho$, in order to capture points in $X$. This is however not true and we will see that $$\underline{\lim}_{\lambda(I)\rightarrow0}\frac{\lambda\left(\{x:m(x)>m(\rho)-\delta\}\bigcap I\right)}{\lambda\left(I\right)}=0.$$ In order to overcome this obstacle we prove the following Proposition, which constitutes the main ingredient of Theorem \ref{Flexibility of the extremal lattice}. In a very informal way we show that if we pick some interval $I$ around $\rho$, then either $99\%$ of the points in $I$ are in $B_{\epsilon}(\rho)\cap E^{\eta}(\rho)$ or there exists some subinterval $I'\subset I$ containing $\rho$, which has at least $\frac{\epsilon}{100}\%$ the size of $I$, such that $99,99\%$ of the points in $I'$ are inside $B_{\epsilon}(\rho)\cap E^{\eta}(\rho).$ 
\begin{thm}[Structural Theorem]\label{Positive Proportion}
    Let $\rho\in\R$ with $E(\rho)<\infty$ and $\epsilon>0$. There exists some constant $C>0$ such that for every pair of positive parameters $(\tau_1,\tau_2)$, there exists $\delta>0$ such that for every interval $I$ of diameter less than $\delta$ and containing $\rho$, we have that one of the following two  holds: \begin{enumerate}
        \item (Many $B_{\epsilon}(\rho)\cap E^{\eta}(\rho)$-points)\\Either $$\lambda(I\cap B_{\epsilon}(\rho)\cap E^{\eta}(\rho))\geq (1-\tau_1)\lambda(I).$$
        \item (Subinterval with higher concentration of $B_{\epsilon}(\rho)\cap E^{\eta}(\rho)-$points)\\ Or there exists some subinterval $I'\subset I$ containing $\rho$ such that $$\lambda(I')\geq C\tau_1\epsilon\lambda(I)$$ and $$\lambda(I'\cap B_{\epsilon}(\rho)\cap E^{\eta}(\rho))>\lambda(I')(1-\tau_2).$$
        
    \end{enumerate} 
\end{thm}
For our applications, the reader is encouraged to consider $0<\tau_2\ll\tau_1\ll1.$
\begin{proof}
Let $\delta>0$ to be chosen later and let $I$ be some interval of size $\delta$ containing $\rho.$ Let $\pi_N(x)$ denote the function $$x\rightarrow\alpha_N(x), x\in\R\setminus\Q.$$ Even though we define this function on the irrational numbers, this is not a problem as our arguments are measure theoretic. Define $N_0$ to be the smallest integer such that the set $\pi_{N_0}(I)$ is not a singleton. Note that as a function of $\delta$, $N_0$ tends to infinity. We now take cases: \\ \\
\underline{\textbf{Case I}} $\#\pi_{N_0}(I)\geq 3:$ \\ \\ 
Suppose $\pi_{N_0}(I)$ consists of the integers in $[h,M]$, with $M-h\geq2$. Since $\rho\in I$, we have $h\leq\alpha_{N_0}(\rho).$ Let $$I'=\left\{x\in\R:\,\,\alpha_{N_0}\in \left[h,\min\left((1-\epsilon)^{-1}\left(Q^{n-2}_{N-1}(\rho)m^{-1}(\rho)+1\right)-1,(1+\epsilon)a_N(\rho)\right)\right]\right\}\bigcap I.$$  Clearly we have that $I'\subset I$ and $\rho\in I',$ since by definition of $m(\rho)$ we know that $$\alpha_{N_0}(\rho)<Q_{N_0-1}^{n-2}m(\rho)^{-1}+1.$$ Remember that $S_{\rho,\epsilon,N,h,\eta}$ is the set of real numbers $x$ such that $$\alpha_i(x)\begin{cases}
=\alpha_i(\rho) &\text{for $i<N$ },\\\in \left[h,\min\left((1-\epsilon)^{-1}\left(Q^{n-2}_{N-1}(\rho)m^{-1}(\rho)+1\right)-1,(1+\epsilon)a_N(\rho)\right)\right] & i=N, \\<{Q_{i-1}^{\eta}(x)}&\text{for all $i>N.$}\end{cases}$$ By Lemma \ref{S-Lemma}, we have that $$S_{\rho,\epsilon,N_0,h,\eta}\subset B_{\epsilon}(\rho)\cap E^{\eta}(\rho),$$ by taking $\delta$ small enough. By Lemma \ref{Measure of Cutting} we further have that $$\lambda(S_{\rho,\epsilon,N_0,h,\eta}\cap I')=\lambda(I')(1+o_{\delta}(1)),$$ and therefore we can clearly choose $\delta$ so that $\lambda(S_{\rho,\epsilon,N_0,h,\eta}\cap I')\geq\lambda(I')(1-\tau_2)$. We now show that $$\frac{\lambda(I')}{\lambda(I)}>C\epsilon,$$ for some absolute constant $C.$ Since $\pi_{N_0}(I)$ consists of the integers in $[h,M]$, it is clear that if we define $$I''=\left\{x\in I:\,\,\alpha_{N_0}\in \left[h+1,\min\left((1-\epsilon)^{-1}\left(Q^{n-2}_{N_0-1}(\rho)m^{-1}(\rho)+1\right)-1,(1+\epsilon)a_N(\rho),M-1\right)\right]\right\},$$ then $I''\subset I'$ and by Lemma \ref{Continued Fraction Distribution}, we have \begin{multline*}\frac{\lambda(I'')}{\lambda(I)}\geq \frac{A}{B}\left( \frac{\frac{1}{h}-\frac{1}{\min\left((1-\epsilon)^{-1}\left(Q^{n-2}_{N_0-1}(\rho)m^{-1}(\rho)+1\right)-1,(1+\epsilon)a_N(\rho),M-1\right)} }{ \frac{1}{h}-\frac{1}{M}}\right)\gg\\  \left( 1-\frac{h}{ (1-\epsilon)^{-1}(Q^{n-2}_{N_0-1}(\rho)m^{-1}(\rho)+1)-1}\right)\gg  \left(1-\frac{\alpha_{N_0}(\rho)}{(1-\epsilon)^{-1}( Q^{n-2}_{N_0-1}(\rho)m^{-1}(\rho)+1)-1} \right)\geq\epsilon.  \notag\end{multline*} Therefore, $I'$ satisfies the second possible conclusion.  \\ \\ \underline{\textbf{Case II}} $\#\pi_{N_0}(I)=2$\\ 

Assume $\pi_{N_0}(I)=\{k,k+1\}$ and set $S_i=I\cap \pi_{N_0}^{-1}(i).$ Without loss of generality we can assume that $\rho$ is one of the endpoints of the interval $I$. We further assume that $\rho\in S_{k+1}.$ The case $\rho\in S_k$ can be treated similarly. For simplicity we write $\pi_{N_0+1}(\rho)=\alpha_{N_0+1}(\rho)=s$. If $S_{k+1,s}$ is the interval with endpoints  $\rho$ and the rational number $$[\alpha_0(\rho),\alpha_1(\rho),...,\alpha_{N_0}(\rho),\infty]$$ then $$S_{k+1,s}\subset I, \, \lambda(S_{k+1,s})\sim \frac{1}{sQ_{N_0}^2}$$ By applying Case I of our proof for the point $\rho$ and the interval $S_{k+1,s}$ we see that we can find $S'_{k+1,s}\subset S_{k+1,s}$ containing $\rho$, with  $$\frac{\lambda(S'_{k+1,s})}{\lambda(S_{k+1,s})}>C\epsilon$$ and $$\frac{\lambda(S'_{k+1,s}\cap B_{\epsilon}(\rho))}{\lambda(S'_{k+1,s})}>(1-\tau_2).$$ 
We now consider $S_k$. Define $u$ the smallest possible natural number such that, if we define $S_{k,1,u}$ the set of real numbers with $$\alpha_i(x)=\begin{cases}
\alpha_i(\rho) &\text{for $< N_0$ },\\k  & i=N_0,\\1  & i=N_0+1, \\\geq u&i=N_0+2,\end{cases},$$ we then have $S_{k,1,u}\subset I.$ By Lemma \ref{Continued Fraction Distribution}: 
$$\lambda(S_{k,1,u})\sim \frac{1}{uQ_{N_0}^2} \text{ and by minimality of $u$: }\lambda(S_{k})\sim\lambda(S_{k,1,u}).$$
If $s\leq u$ we are done, since then $\lambda(S_{k+1,s})\sim \lambda(I)$ and thus $\lambda(S'_{k+1,s})>C'\epsilon\lambda(I)$ for some absolute constant $C'.$ Assume, thus, that $u<s.$ By similar considerations as previously we have that:$$\frac{\lambda(S_{k,1,u}\cap B_{\epsilon}(\rho)\cap E^{\eta}(\rho))}{\lambda(S_{k,1,u})}= (1+o_{\delta})\frac{\frac{1}{u}-\frac{1}{s}}{\frac{1}{u}}=(1-\frac{u}{s})(1+o_{\delta}).$$ Taking everything together, we have that $$\frac{\lambda(I\cap B_{\epsilon}(\rho)\cap E^{\eta}(\rho))}{\lambda(I)}\gg\frac{\lambda\left(S_{k,1,u}\cap B_{\epsilon}(\rho)\cap E^{\eta}(\rho)\right)}{\lambda(S_{k,1,u})+\lambda(S_{k+1,s})}=(1+o_{\delta})\frac{\frac{1}{u}-\frac{1}{s}}{\frac{1}{u}+\frac{1}{s}}= 1-(1+o_{\delta})\frac{2}{1+\frac{s}{u}}.$$
 Therefore, if $\frac{s}{u}+1\geq 2\tau_1^{-1}$ we are in the case of ``Many $B_{\epsilon}(\rho)\cap E^{\eta}(\rho)-$points". If $\frac{s}{u}+1< 2\tau_1^{-1},$ then $$\frac{\lambda(S'_{k+1,s})}{\lambda(I)}=\frac{\lambda(S'_{k+1,s})}{\lambda(S_{k+1,s})}\frac{\lambda(S_{k+1,s})}{\lambda(I)}\geq C\epsilon\frac{\frac{1}{s}}{\frac{1}{s}+\frac{1}{u}}=C\epsilon\frac{1}{1+\frac{s}{u}}>\frac{1}{2}C\tau_1\epsilon,$$ and we are in the case of a ``Subinterval with higher concentration of $B_{\epsilon}(\rho)\cap\nolinebreak E^{\eta}(\rho)-$points" inside the subinterval $S'_{k+1,s}$.
\end{proof}

For any real number $z$, define \begin{align}&\Pi_z:\SL_2(\R)\rightarrow \R, \notag\\ &
\Pi_z(T)=T(z). \notag\end{align}

Theorem \ref{Flexibility of the extremal lattice} is an immediate corollary of the following:

\begin{prop}\label{algorithm}
    Let $\rho_i,$ $i\in\{1,2,...,k\}$ be $k$ distinct points in $\R.$ Then for any $\eta\in(0,1)$ we have $$\overline{\lim}_{\epsilon\rightarrow0}\frac{\mu(\bigcap_{i=1}^k\Pi_{\rho_i}^{-1}\left(B(\rho_i,\epsilon)\bigcap B_{\epsilon}(\rho_i)\bigcap E^{\eta}(\rho_i)\right))}{\mu(\bigcap_{i=1}^k\Pi_{\rho_i}^{-1}\left(B(\rho_i,\epsilon)\right))}=1.$$
\end{prop}
\begin{proof}
 
 Let $\epsilon>0$ and $\tau_1,\tau_2$ to be chosen later: Pick some $\delta$ that satisfies the assumptions of Proposition \ref{Positive Proportion} for all the $\rho_i's$. For any interval $I$ and $i,j\in\{1,2...,k\},$ we define $$\Pi_{i,j}(I)=\Pi_{\rho_i}(\Pi_{\rho_j}^{-1}(I)\cap C(\delta)),$$ where $C(\delta)$ denotes some cube of diameter $\delta$ around the identity matrix of $\SL_2(\R)$, when identifying this space locally with $\R^3.$  We now construct an algorithm that will give us our transformations $T\in\bigcap_{i=1}^k\Pi_{\rho_i}^{-1}\left( B(\rho_i,\epsilon)\bigcap B_{\epsilon}(\rho_i)\bigcap E^{\eta}(\rho_i)\right)$. Let $I_0$ any interval containing $\rho_1$ with diameter less than $\delta.$ Therefore, $\Pi_{\rho_1}^{-1}(I_0)\cap C(\delta)$ is an open neighborhood of the identity matrix in $\SL_2(\R),$ and $\Pi_{2,1}(I_0)=\Pi_{\rho_2}(\Pi_{\rho_1}^{-1}(I_0)\cap C(\delta))$ is an interval containing $\rho_2.$ We say this interval is of Type I if it satisfies the first conclusion of Proposition \ref{Positive Proportion} and Type II if it satisfies the second. If it satisfies both, we call it Type I.
 \begin{itemize}
     \item Set $L_0=\emptyset.$ If $\Pi_{1,1}(I_0)$ is of Type I, then consider $\Pi_{2,1}(I_0)$. If $\Pi_{2,1}(I_0)$ is of Type I, then consider $\Pi_{3,1}(I_0).$ If all of them up to $\Pi_{k,1}(I_0)$ are Type I, terminate the procedure.
     \item Suppose $i_0$ is the smallest natural with $\Pi_{i_0,1}(I_0)$ of Type II. Then, we can find by Proposition \ref{Positive Proportion}, some subinterval $I'_1\subset \Pi_{i_0,1}(I_0)$, such that $$\lambda(I'_1)\geq C\tau_1\epsilon\lambda(\Pi_{i_0,1}(I_0)) \text{ and } \lambda(I'_1\cap B_{\epsilon}(\rho_{i_0})\cap E^{\eta}(\rho_{i_0}))\geq (1-\tau_2)\lambda(I'_1).$$ Define $I_1=\Pi_{1,i_0}(I'_1)\cap I_0$ and $L_1=L_0\cup\{i_0\}.$
     \item We repeat the first step with $I_1$ instead of $I_0$ and skipping the numbers in $L_1:$ \\ \\ 
     Let $i_1\notin L_1$ the smallest number such that $\Pi_{i_1,1}(I_1)$ is Type II. If there is none, terminate the procedure. If there is, then by Proposition \ref{Positive Proportion}, we can find some subinterval $I'_2\subset \Pi_{i_1,1}(I_1)$, such that $$\lambda(I'_2)\geq C\tau_1\epsilon\lambda(\Pi_{i_1,1}(I_1)) \text{ and } \lambda(I'_2\cap B_{\epsilon}(\rho_{i_1})\cap E^{\eta}(\rho_{i_1}))\geq (1-\tau_2)\lambda(I'_2).$$ Define $I_2=\Pi_{1,i_1}(I'_2)\cap I_1$ and $L_2=L_1\cup\{i_1\}.$
     \item Terminate the algorithm at $\mathbb{T}-$th step if $L_{\mathbb{T}}=\{1,2...,k\}.$
 \end{itemize}Notice that after at most $k$ steps we obtain an interval $I_f$ around $\rho_1$ and some subset $L$ of $\{1,2,...,k\}$ such that \begin{enumerate}
     \item For all $i\in L$: $$\lambda(\Pi_{i,1}(I_f)\cap B_{\epsilon}(\rho_i)\cap E^{\eta}(\rho_i))\geq (1-(C\tau_1\epsilon)^{-k}\tau_2)\lambda(\Pi_{i,1}(I_f)),$$ by possibly increasing the value of the constant $C$. Here we have used that $\lambda(\Pi_{i,1}(I))\sim\lambda(I)$ for some small interval $I$ around $\rho_1$.
     \item For all $i\in\{1,2,...,k\}\setminus L$: $$\lambda(\Pi_{i,1}(I_f))\cap B_{\epsilon}(\rho_i)\cap E^{\eta}(\rho_i))\geq (1-\tau_1)\lambda(\Pi_{i,1}(I_f)).$$
 \end{enumerate}By identifying small enough neighborhoods of $\SL_2(\R)$ around the identity with the space $\R^3$, it is not hard to see that for every $z\in\R$, there exist positive constants $m,M$ such that for $\delta$ small enough and for every measurable set\\ $B\subset (z-\delta,z+\delta),$ we have that  \begin{align}&m\mu\left(\Pi^{-1}_z(B)\cap C(\delta)\right) <\lambda\left(B\right)< M\mu\left(\Pi^{-1}_z(B)\cap C(\delta)\right).\label{Measure Similarity} \end{align} 
Therefore, we have $$\frac{\mu(\Pi_{1}^{-1}(I_f)\cap C(\delta)\cap \Pi_{i}^{-1}(B_{\epsilon}(\rho_i)\cap E^{\eta}(\rho_i)))}{\mu(\Pi_{1}^{-1}(I_f)\cap C(\delta))}> \begin{cases}
    1-M(C\tau_1\epsilon)^{-k}\tau_2 & i\in L \\
    1-M\tau_1 & i\notin L.
\end{cases}$$

Pick $\tau_1<\frac{\epsilon}{M}$ and $\tau_2<\frac{(C\tau_1\epsilon)^{k}\epsilon}{M}$. We deduce that $$\frac{\mu(\Pi_{1}^{-1}(I_f)\cap C(\delta)\cap \Pi_{i}^{-1}(B_{\epsilon}(\rho_i)\cap E^{\eta}(\rho_i)))}{\mu(\Pi_1^{-1}(I_f)\cap C(\delta))}>1-\epsilon.$$ 
\end{proof}

We have established the existence of $\epsilon$-almost extremal lattices for every $\epsilon>0$ and therefore proved the absence of a spectral gap for the spectrum of all binary forms of degree $n\geq3.$ The lattices we constructed have the additional property that all of the roots of $P\circ\Lambda$ lie in $E_2$ and hence
\begin{corollary}
    Let $P$ denote a homogeneous binary form of degree $n\geq3$ and non-zero discriminant. Then, for every $\epsilon>0$, there exists a finitely minimized $\epsilon-$almost extremal lattice.
\end{corollary}
Given our result, the following question regarding the existence of finitely minimized extremal lattices is natural. 

\textbf{Question:} \textit{Does the star body of every homogeneous binary form of non-zero discriminant possess a finitely minimized extremal lattice?}

\section{The Measure of $\Spec(P)$} \label{The Measure of Spec(P)}
The results of the previous sections imply that we can find almost extremal lattices that have good diophantine properties with respect to our form $P$. In this section, and specifically in Theorem \ref{Density theorem}, we show how we can generate a spectrum of full measure using diagonal perturbations of such lattices. This is, of course, a weaker measure theoretic version of Theorem \ref{Main Theorem}. We begin with the following proposition regarding diagonal perturbations. 

\begin{prop}\label{Main Theorem measure}
Let $P(x,y)=\prod\limits_{i=1}^k\left(x-y\rho_i\right)D(x,y)$ a binary form of degree $n\geq3$ of non-zero discriminant, with $k$ real roots $\rho_{i},$ where $ i\in\{1,2,...,k\}$ and positive definite part denoted by $D$. Denote by $\frac{P_N}{Q_N}$ the convergents of $\rho_1$ and define $\theta_N$ to be the first point on the left of $\frac{P_N}{\rho_1Q_N}$ such that $$\left\vert P\circ\Delta_{\theta_N}(P_N,Q_N)\right\vert=\inf_{(x,y)\in\Z^2\backslash\underline{0}}\left\vert P\left(x,y\right)\right\vert,$$ where $$\Delta_{\theta}=\begin{pmatrix}
    \sqrt{\theta}& 0\\
    0 & \sqrt{\frac{1}{\theta}}
\end{pmatrix},\theta\in\R_{+}.$$ There exists some $M>0$, that depends continuously on the coefficients of $P$, such that for every $\theta\in[\theta_N,\frac{P_N}{\rho_1Q_N}]$ one of the following holds:
\begin{enumerate}
\item $P\circ\Delta_{\theta}$ is finitely minimized by $(P_N,Q_N)$.
\item There exists $(x_c,y_c)\in\Z\times\mathbb{N}$ with $y_c>\frac{1}{M}Q_N^{\frac{5}{4}}$ and $$\left\vert\frac{x_c}{y_c}-\theta{\rho_i}\right\vert<\frac{M}{y_c^{n}},$$ for some $i\in\{1,2,...,k\}.$
\item $P\circ\Delta_{\theta}$ is finitely minimized by some 
$(x_c,y_c)\in\Z\times\mathbb{N}$ with $y_c<MQ_N^{\frac{5}{7}}$ and $$\left\vert\frac{x_c}{y_c}-{\rho_i}\right\vert<\frac{M}{y_c^{2+\frac{4}{5}}},$$ for some $i\in\{1,2,...,k\}.$
\item There exists $(x_c,y_c)\in\Z\times\mathbb{N}$ with $y_c<MQ_N^{\frac{5}{4}}$ and $$\left\vert\frac{x_c}{y_c}-\frac{\rho_i}{\rho_1}\frac{P_N}{Q_N}\right\vert<\frac{M}{(y_cQ_N)^{\frac{5}{4}}},$$ for some $i\in\{2,...,k\}.$

\end{enumerate}
\end{prop}

\begin{proof}
Assume $N$ is odd for convenience, so that $\frac{P_N}{Q_N}>\rho_1$. We have the formula $$P\circ\Delta_{\theta}(x,y)={\theta}^{-\frac{n}{2}}\prod_i(x-\theta\rho_i y).$$ It is clear that $$P\circ \Delta_{\frac{P_N}{\rho_1Q_N}}(P_N,Q_N)=0.$$ Since $$\pdv{P\circ\Delta_{\theta}(P_N,Q_N)}{\theta}\biggr\vert_{\frac{P_N}{\rho_1Q_N}}=C(1+o_N(1))Q_N^n,$$ for some explicit non-zero constant $C$, we have $\theta_N=\frac{P_N}{\rho_1Q_N}+O(\frac{1}{Q_N^n})$.  Consider the interval $$I_N=[\theta_N,\frac{P_N}{\rho_1Q_N}].$$ By an analogue of Lemma \ref{Finitely Minimized}, we know that for almost all $\theta$, $P\circ\Delta_{\theta}$ is finitely minimized. If for some $\theta\in I_N,$ $P\circ\Delta_{\theta}$ is not finitely minimized, or minimized by some $(x_c,y_c)$ with $y_c\geq Q_N^{\frac{5}{4}},$ then we are clearly in the second possible case. Assume hence that $P\circ\Delta_{\theta}$ is minimized by some $y_c\neq Q_N$ with $y_c<Q_N^{\frac{5}{4}}$. By the intermediate value theorem there should be an intersection of the curves: $$\theta\rightarrow P\circ\Delta_\theta(P_N,Q_N)$$ and
$$\theta\rightarrow P\circ\Delta_\theta(x_c,y_c),$$ at some $\theta_c\in I_N$. Therefore, we have: 
\begin{align}
Q_N^n\prod\limits_{i=1}^k\left(\frac{P_N}{Q_N}-\theta_c\rho_i\right)D\circ\Delta_{\theta_c}(\frac{P_N}{Q_N},1)&=y_c^n\prod\limits_{i=1}^k\left(\frac{x_c}{y_c}-\theta_c\rho_i\right)D\circ\Delta_{\theta_c}(\frac{x_c}{y_c},1). \notag
\end{align}
Since $\theta_c=\frac{1}{\rho_1}\frac{P_N}{Q_N}+O(\frac{1}{Q_N^n})$, we have
\begin{align}
y_c^n\prod\limits_{i=1}^k\left(\frac{x_c}{y_c}-\frac{\rho_i}{\rho_1}\frac{P_N}{Q_N}\right)+O\left(\frac{y_c^n}{Q_N^n}\right)&=O(1), \notag
\notag \end{align} and thus 

\begin{align}
\prod\limits_{i=1}^k\left(\frac{x_c}{y_c}-\frac{\rho_i}{\rho_1}\frac{P_N}{Q_N}\right)&=O(\frac{1}{Q_N^n})+O(\frac{1}{y_c^n}). \label{Basic General Inequality}
\end{align}

\noindent If there exists some rational number $\frac{x}{y}$ with $y\ll Q_N^{\frac{5}{4}}$ and $$\left\vert\frac{x}{y}-\frac{\rho_i}{\rho_1}\frac{P_N}{Q_N}\right\vert\ll \frac{1}{{(yQ_N)}^{\frac{5}{4}}},$$ for some $i\in\{2,3,...,k\}$, then we are in the fourth possible case and hence assume that for $y\ll Q_N^{\frac{5}{4}},$ we have\begin{align}\left\vert\frac{x_c}{y_c}-\frac{\rho_i}{\rho_1}\frac{P_N}{Q_N}\right\vert&\gg \frac{1}{{(yQ_N)}^{\frac{5}{4}}}. \label{General inequality}\end{align} 
Therefore, $$\frac{1}{(y_cQ_N)^{\frac{5}{4}}}=O(\frac{1}{Q_N^n})+O(\frac{1}{y_c^n}),$$
which gives $$\min(y_c,Q_N)^{n-\frac{5}{4}}\ll \max(y_c,Q_N)^{\frac{5}{4}},$$
    or equivalently 
\begin{align}  
\min(y_c,Q_N)^{\frac{4n}{5}-1}&\ll\max(y_c,Q_N).\label{Main general inequality}
\end{align}
We now consider individually each possibility for the $\min(y_c,Q_N)$: 
\begin{itemize}
    \item If $\min(y_c,Q_N)=y_c$, then by Eq. \ref{Main general inequality}, $y_c\ll Q_N^{\frac{5}{4n-5}}\ll Q_N^{\frac{5}{7}}.$ 
    Hence, we obtain
    $$
        \frac{1}{y_c^n}\gg\left\vert\frac{x_c}{y_c}-\frac{\rho_i}{\rho_1}\frac{P_N}{Q_N}\right\vert\geq \left\vert\frac{x_c}{y_c}-\rho_i\right\vert-\left\vert\frac{\rho_i}{\rho_1}(\frac{P_N}{Q_N}-\rho_1)\right\vert \gg \notag   \left\vert\frac{x_c}{y_c}-\rho_i\right\vert-\frac{1}{Q_N^2}, $$ and thus $$\left\vert\frac{x_c}{y_c}-\rho_i\right\vert\ll \frac{1}{y_c^{2+\frac{4}{5}}}.$$ 
\item If $\min(y_c,Q_N)=Q_N$, we have $Q_N\ll y_c^{\frac{5}{4n-5}}$ and $y_c\gg Q_N^{\frac{7}{5}}\gg Q_N^{\frac{5}{4}}$, since $n\geq3.$ 
\end{itemize}
\end{proof}
As we will see, the usefullness of Proposition \ref{Main Theorem measure} lies in the fact that if $\rho_i$ have ``good diophantine properties", then the possible conclusions $(2),(3),(4)$ of the proposition, occupy only a $o_N(1)\%$ proportion of $I_N,$ providing a way to construct curves of almost full spectral measure near an extremal lattice. We begin with conclusion $(4)$:

\begin{lemma} \label{D has measure 0}
Let $\rho_1\in\R$ and $\frac{P_N}{Q_N}$ its convergents. Let $D_{\rho_1}$ denote the set of all $\rho\in\R$, such that for some $M>0$, there exist infinitely many $N\in\mathbb{N}$ and integers $x_c,y_c$ satisfying \begin{itemize}
     \item $y_c<MQ_N^{\frac{5}{4}}$,
     \item $$\left\vert\frac{x_c}{y_c}-\frac{\rho}{\rho_1}\frac{P_N}{Q_N}\right\vert<\frac{M}{(y_cQ_N)^{\frac{5}{4}}}.$$
 \end{itemize}   Then $$\lambda(D_{\rho_1})=0.$$   
\end{lemma}
\begin{proof}
We have  
$$D_{\rho_1}=\bigcup\limits_{M\in\mathbb{N}}\bigcap\limits_{N_0\in\mathbb{N}}\bigcup\limits_{N>N_0}\bigcup\limits_{\substack{(X,Y)\in\Z\times\mathbb{N},\\Y<MQ_N^{\frac{5}{4}}}}I(N,X,Y),$$where $$I(N,X,Y)=\left(\rho_1\frac{Q_N}{P_N}\frac{X}{Y}-{\rho_1}\frac{Q_N}{P_N}\frac{M}{(YQ_N)^{\frac{5}{4}}},\rho_1\frac{Q_N}{P_N}\frac{X}{Y}+{\rho_1}\frac{Q_N}{P_N}\frac{M}{(YQ_N)^{\frac{5}{4}}}\right).$$
 We therefore have that for any interval $[a,b]$, there exist some constants $A,B,$ such that $$D_{\rho_1}\cap[a,b]=\bigcup\limits_{M\in\mathbb{N}}\bigcap\limits_{N_0\in\mathbb{N}}\bigcup\limits_{N>N_0}\bigcup\limits_{\substack{AY<X,\\X<BY}}\bigcup\limits_{\substack{(X,Y)\in\Z\times\mathbb{N},\\Y<MQ_N^{\frac{5}{4}}}}I(N,X,Y).$$To show that this set is of measure zero it suffices by classical measure theory to show that for any $M>0$, 
$$\sum\limits_{N\in \mathbb{N}} \lambda\left(\bigcup\limits_{\substack{AY<X,\\X<BY}}\bigcup\limits_{\substack{Y\in\mathbb{N},\\Y<MQ_N^{\frac{5}{4}}}}I(N,X,Y)\right)<\infty.$$By the trivial bounds, we have
\begin{align}
    &\sum\limits_{N\in \mathbb{N}} \lambda\left(\bigcup\limits_{\substack{AY<X,\\X<BY}}\bigcup\limits_{\substack{Y\in\mathbb{N} \\ Y<MQ_N^{\frac{5}{4}}}}I(N,X,Y)\right)\ll\sum\limits_{N\in \mathbb{N}} \sum\limits_{Y<MQ_N^{\frac{5}{4}}}\frac{1}{Y^{\frac{1}{4}}Q_N(\rho)^{\frac{5}{4}}}\ll \notag \\ &\sum\limits_{N\in \mathbb{N}}\frac{Q_N^{\frac{3}{4}\frac{5}{4}}}{Q_N^{\frac{5}{4}}}\ll \sum\limits_{N\in\mathbb{N}}\phi^{-\frac{5}{16} N}, \notag
\end{align} The result follows.
\end{proof}
\begin{thm}  
\label{Density theorem}
    Let $P$ a binary form of degree $n\geq3$ of non-zero discriminant. Then, $\Spec(P)$ is a set of full measure inside $[0,M_P].$

\end{thm}
\begin{proof}
 If $\Z^2$ is not an extremal lattice for $P$, we can simply conjugate by some $T\in\SL_2(\R)$ and arrange that. Hence, assume that $m(P)=M_P$. Let $\epsilon>0$. By Theorem \ref{Flexibility of the extremal lattice}, Lemma \ref{D has measure 0} and the fact that $E_2$ has full measure, we can find $T_{\epsilon}\in\SL_2(\R),$ such that \begin{itemize}
\item $\vert\vert T_{\epsilon}-id_2\vert\vert<\epsilon,$
    \item $m(P\circ T_{\epsilon})>m(P)-\epsilon$, 
    \item $\rho_i\notin D_{T_{\epsilon}(\rho_1)}$, for $i\in\{1,2,...,k\}$,
    \item The roots of $P\circ T_{\epsilon}$ lie in $E_2$.
\end{itemize} Denote by $I_{N,\epsilon}$ the interval defined in Proposition \ref{Main Theorem measure}, corresponding to the form $P\circ T_{\epsilon}$ and the root $T(\rho_1)$. It follows that there exists some $M_{\epsilon}>0,$  such that for $N$ large enough and $\theta\in I_{N,\epsilon}$, one of the following holds: \begin{enumerate}
    \item $P\circ T_{\epsilon}\circ\Delta_{\theta}$ is minimized by $(P_N(T_{\epsilon}(\rho_1)),Q_N(T_{\epsilon}(\rho_1))),$ 
    \item $P\circ T_{\epsilon}\circ\Delta_{\theta}$ is minimized by $(x_c,y_c)$, for some $(x_c,y_c)\in\Z\times\mathbb{N}$ with $y_c<M_{\epsilon}$,
    \item there exist $(x_c,y_c)\in\Z\times\mathbb{N}$ with $y_c>\frac{1}{M_{\epsilon}}Q_N(T_{\epsilon}(\rho_1))^{\frac{5}{4}}$ and $$\left\vert\frac{x_c}{y_c}-\theta{\rho_i}\right\vert<\frac{M_{\epsilon}}{y_c^{n}},$$ for some $i\in\{1,2,...,k\}.$
    \end{enumerate}
    Here, we used that $\rho_i\in E_2$ along with the fact that the constant $M$ behaves continuously on the roots $\rho_i$. Denote $S_i$ the set of those $\theta\in I_{N,\epsilon}$ for $i=1,2,3$ in each of the above cases respectively. We now prove that $$\lambda(S_1)=(1+o_N(1))\lambda(I_{N,\epsilon}).$$
    It clearly suffices to show that $$\lambda(S_2)=o_N(1)Q_N^{-n}(T_{\epsilon}(\rho_1))\text{ and }\lambda(S_3)=o_N(1)Q_N^{-n}(T_{\epsilon}(\rho_1)).$$
    Notice that for some fixed $(x_0,y_0)\in\Z\times\mathbb{N}$, the set $$\{\theta\in I_{N,\epsilon}: \,P\circ T_{\epsilon}\circ\Delta_{\theta}\text{ is minimized by $(x_0,y_0)$}\}$$ has measure $o_N(1)\frac{1}{Q_N^n(T_{\epsilon}(\rho_1))}$ and therefore, $\lambda(S_2)=o_N(1)\frac{1}{Q_N^n(T_{\epsilon}(\rho_1))}$. As for $S_3,$ fix some $i\in\{1,2,...,k\}$ and let $S_{3,i}\subset I_{N,\epsilon}$ denote the set of those $\theta$, such that there exist $(x_c,y_c)\in\Z\times\mathbb{N}$ with $y_c>\frac{1}{M_{\epsilon}}Q_N(T_{\epsilon}(\rho_1))^{\frac{5}{4}}$ and $$\left\vert\frac{x_c}{y_c}-\theta{\rho_i}\right\vert<\frac{M_{\epsilon}}{y_c^{n}}.$$ We have $$\lambda(S_{3,i})\leq\sum\limits_{\substack{\frac{X}{Y}\in\mathbb{Q}\cap I_{N,\epsilon}, \\  Y > \frac{1}{M_{\epsilon}}Q_N(T_{\epsilon}(\rho_1))^{\frac{5}{4}}}}\frac{1}{Y^n}.$$
For some $Y\in \mathbb{N},$ denote by $N(Y)$ the cardinality of the set $$\left\{X\in\Z:\, \frac{X}{Y}\in I_{N,\epsilon}\right\}.$$ For $N(Y)$, we have the trivial bound $N(Y)\ll \frac{Y}{Q_N^n(T_{\epsilon}(\rho_1))}.$ We thus get the bound 
$$\lambda(S_{3,i})\ll\sum\limits_{Y > \frac{1}{M_{\epsilon}}Q_N(T_{\epsilon}(\rho_1))^{\frac{5}{4}}} \frac{N(Y)}{Y^n}\ll\frac{1}{Q_N(T_{\epsilon}(\rho_1))^n}\sum\limits_{Y > \frac{1}{M_{\epsilon}}Q_N(T_{\epsilon}(\rho_1))^{\frac{5}{4}}} \frac{1}{Y^{n-1}}\ll \frac{1}{Q_N(T_{\epsilon}(\rho_1))^{n+\frac{5(n-2)}{4}}}.$$

\noindent However, for $\theta\in S_1,$ $$\left\vert P\circ T_{\epsilon}\circ\Delta_{\theta}(P_N,Q_N)\right\vert\in\Spec (P),$$ and hence $$\lambda(\Spec(P))\geq m(P\circ T_{\epsilon})-o_N(1).$$ By taking $N\rightarrow\infty$, we have proved that $$\lambda(\Spec(P))\geq m(P\circ T_{\epsilon})\geq m(P)-\epsilon.$$ The theorem follows by taking $\epsilon\rightarrow0$.
\end{proof}

\section{Fixed point perturbations and intervals in the spectrum} \label{Fixed point perturbations and intervals in the spectrum}
The machinery we have built so far, coupled with a generalization of the classical Steinhaus theorem (\cite{Steinhaus},\cite{erdos1955partitions}), is already enough to prove Theorem \ref{D>0}.
\begin{lemma}[General Steinhaus Theorem]\label{General Steinhaus}
    Fix $\epsilon,m,M>0$. There exists some $\delta$, such that for every $a_2>a_1>0,b_2>b_1>0,c_2>c_1>0$, every $C^{\infty}$ function $f:[a_1,a_2]\times [b_1,b_2]\times [c_1,c_2]\rightarrow \R$ satisfying  $$m<\vert\pdv{u}{a}(a_0,b_0,c_0)\vert,\vert\pdv{f}{b}(a_0,b_0,c_0)\vert,\vert\pdv{f}{c}(a_0,b_0,c_0)\vert<M,$$ for all $(a_0,b_0,c_0)\in [a_1,a_2]\times [b_1,b_2]\times [c_1,c_2]$, has the following property: For any $X_1\times X_2\times X_3\subset [a_1,a_2]\times [b_1,b_2]\times [c_1,c_2]$ measurable set with \begin{align*}
    &\lambda(X_1)\geq(1-\delta)\lambda([a_1,a_2]),\\ & \lambda(X_2)\geq(1-\delta)\lambda([b_1,b_2]),\\ & \lambda(X_3)\geq(1-\delta)\lambda([c_1,c_2]),\end{align*}we have $$f\left(\left[a_1(1+\epsilon),a_2(1-\epsilon)\right]\times \left[b_1(1+\epsilon),b_2(1-\epsilon)\right]\times \left[c_1(1+\epsilon),c_2(1-\epsilon)\right]\right)\subset f(X_1\times X_2\times X_3).$$
\end{lemma}
\begin{proof}
Fix some $h\in f([a_1,a_2]\times [b_1,b_2]\times [c_1,c_2]),$ and assume that for some $a,b,c\in[a_1,a_2]\times [b_1,b_2]\times [c_1,c_2]$, we have $h=f(a,b,c).$ Then by the implicit function theorem, we can find some $C^{\infty}$ function $$g_h: I_h\rightarrow[b_1,b_2]\times [c_1,c_2],$$ where $$I_h=[a_1,a_2]\cap B(a,\min(\frac{m}{M}\min(b-b_1,b_2-b),\frac{m}{M}\min(c-c_1,c_2-c))),$$ such that $f(a,g_h(a))=h.$ Write $\textbf{b}$ for $b_2-b_1$ and $\textbf{c}$ for $c_2-c_1$. We have,
\begin{align*}
    &\lambda_2(g_h(X_1\cap I_h)\cap X_2\times X_3)=\lambda_2(g_h(X_1\cap I_h))+\lambda_2(X_2\times X_3)-\lambda_2(g_h(X_1\cap I_h)\cup X_2\times X_3)\geq \\&
    \lambda_2(g_h(X_1\cap I_h))+(1-\delta)^2(\textbf{b}+\textbf{c})-(\textbf{b}+\textbf{c})=\lambda_2(g_h(I_h))-\lambda_2(g_h(X_1^c\cap I_h))-\delta(2-\delta)(\textbf{b}+\textbf{c})\geq\\ & \lambda_2(g_h(I_h))-M\lambda(X_1^c\cap I_h)-\delta(2-\delta)(\textbf{b}+\textbf{c})\geq
    \lambda_2(g_h(I_h))-M\delta-\delta(2-\delta)(\textbf{b}+\textbf{c}).
\end{align*} We have shown that   $$\lambda_2(g_h(I_h))\geq M\delta+\delta(2-\delta)(\textbf{b}+\textbf{c})\implies g_h(X_1)\cap \left(X_2\times X_3\right)\neq\emptyset,$$ which concludes our proof.
\end{proof}
We are now in position to prove Theorem \ref{D>0}.
\begin{proof}[Proof of Theorem \ref{D>0}] Let $\rho>\chi>\psi$ denote the real roots of $x^3+x^2-2x-1=0$, with $P(x,y)=x^3+x^2y-2xy^2-y^3$ being the extremal form for the lattice $\Z^2$, as proved by Mordell. Consider the following family of binary cubic forms $$P_{\theta_1,\theta_2,\theta_3}(x,y)=c_{\theta_1,\theta_2,\theta_3}(x-\theta_1\rho y)(x-\theta_2\chi y)(x-\theta_3\psi y),$$ where $c_{\theta_1,\theta_2,\theta_3}$ is such that the discriminant of the form is $1$. Denote by $\frac{P_N(\rho)}{Q_N(\rho)}$ the convergents of $\rho$ and define $\theta_N(\rho)$ to be the first point on the left of $\frac{P_N(\rho)}{\rho Q_N(\rho)}$ such that $$\left\vert P_{\theta_N(\rho),1,1}(P_N(\rho),Q_N(\rho))\right\vert=\sqrt[4]{\frac{1}{49}}.$$ We define $I_N^{\rho}=[\theta_N(\rho),\frac{P_N(\rho)}{\rho Q_N(\rho)}]$ and analogously we can define the sets $I_N^{\chi}$, $I_N^{\psi}.$ Similarly to the proof of Theorem \ref{Density theorem}, we can find subsets $\overline{I_N^{\rho}},\overline{I_N^{\chi}},\overline{I_N^{\psi}}$ of proportion $100-o_N(1)\%$ such that for all $(\theta_1,\theta_2,\theta_3)\in \overline{I_N^{\rho}}\times \overline{I_N^{\chi}}\times \overline{I_N^{\psi}}$, we have that $m(P_{\theta_1,\theta_2,\theta_3})$ is minimized by \begin{itemize}
    \item $(P_N(\rho),Q_N(\rho))$ or
    \item $(P_N(\chi),Q_N(\chi))$ or
    \item $(P_N(\psi),Q_N(\psi))$.
\end{itemize}
Hence, by Lemma \ref{General Steinhaus}, we get by taking $N\rightarrow\infty$, that $\Spec(P)$ contains the set $$\left\{\min_{(P_N(\rho),Q_N(\rho)),(P_N(\chi),Q_N(\chi)),(P_N(\psi),Q_N(\psi))}\vert P_{\theta_1,\theta_2,\theta_3}(x,y)\vert:\,(\theta_1,\theta_2,\theta_3)\in I_N^{\rho}\times I_N^{\chi}\times I_N^{\psi} \right\}^{\mathrm{o}},$$ which implies that $$\left(0,\sqrt[4]{\frac{1}{49}}\right)\in\Spec(P).$$ Since we have already shown that $0$ and $\sqrt[4]{\frac{1}{49}}$ are in the spectrum, we have concluded the proof of the theorem.
\end{proof}
The use of the generalization of Steinhaus' theorem relies on the product structure of the family we constructed. This construction is not possible for $n\geq4$ as we have discussed.

\begin{defn}
 Let $P$ denote some binary form of degree $n\geq3$. As usual, we denote its roots by $\rho_i$, $i\in\{1,2,...,n\}$, with $\rho_i$ being real for $i\in\{1,2,...,k\}.$ For $N\in\mathbb{N}$ let $I_N$ as in Proposition \ref{Main Theorem measure}. For $\theta\in I_N$, we define the curve $$\Sigma_{N,\theta,u}=\left\{T\in\SL_2(\R):\, T(\rho_1)=\rho_1,\,\vert\vert T-id_2\vert\vert<1,\,\,\prod_{i\geq2}(\frac{P_N}{Q_N}-\theta T(\rho_i))=u\right\}.$$ \end{defn}

Since $\SL_2(\R)$ is $3$ dimensional, any $T\in\Sigma_{N,\theta,u}$ is completely determined by the value of $T(\rho_2)$. Fix some $i\in\{3,...,n\}$. By solving the equation $$\prod_{i\geq2}(\frac{P_N}{Q_N}-\theta T(\rho_i))=u$$ in the two variables $T(\rho_2),T(\rho_i)$, we obtain a non trivial polynomial equation $$F_{N,\theta,u,i}(T(\rho_2),T(\rho_i))=0$$ unique up to a constant. Taking $N\rightarrow\infty$ and $u\rightarrow\prod_{i\geq2}(\rho_1-\rho_i),$ $F_{N,\theta,u,i}$ converges to some polynomial $F_i$ in the two variables $T(\rho_2)$ and $T(\rho_i)$, unique up to rescaling, which trivially does not reduce to a one variable polynomial. The non-vanishing of the derivative of $F_i$ with respect to the variable $T(\rho_i)$ implies by the implicit function theorem, that $T(\rho_i)$ is a $C^{\infty}$ function of $T(\rho_2)$ on the curve $\Sigma_{N,\theta,u}$. We show that can find many matrices $T\in\Sigma_{N,\theta,u}$, such that the roots of $P\circ T$ have our desired Diophantine properties:

\begin{prop}\label{Sigma subvariety}
    Let $P$ some binary form with real roots denoted by $\rho_i$, $i\in\{1,2,...,k\}$ and let $\eta>0$. Assume that $$\pdv{F_i}{\rho'_2}\vert_{(\rho_2,\rho_i)}\neq 0$$ for $i\in \{3,...,k\}$. For every $\epsilon>0,$ there exists some $\delta>0$, such that for every $N$ large enough, $\theta\in I_N$ and $u$ satisfying $$\vert v-\prod_{i\geq2}(\rho_i-\rho_1)\vert<\delta,$$ there exists some non-trivial open subset $U\subset \Sigma_{N,\theta,v}$ such that $$\mu\left(U\bigcap_{i\geq2}\Pi_i^{-1}(E^{\eta}(\rho_i)\cap B_{\epsilon}(\rho_i)\cap B(\epsilon,\rho_i))\right)\geq (1-\epsilon)\mu\left(U\right),$$ where here $$\mu(U)=\lambda(\left\{T(\rho_2), \,T\in U\right\}).$$    \end{prop}

    \begin{proof}
     Fix some $\epsilon>0.$ By the non-vanishing of the partial derivatives and the implicit function theorem, we can find for $\epsilon,\delta$ small enough, $C^{\infty}$ functions $f_{i,N,\theta,u}$,  defined locally as: $$T(\rho_2)=\rho'_2,\, T\in\Sigma_{N,\theta,v}\iff T(\rho_i)=f_{i,N,\theta,u}(\rho'_2).$$ Note that this implies that the measure $\mu$ is well defined. For example, $$f_{1,N,\theta,u}\equiv \rho_1\text{ and }f_{2,N,\theta,u}\equiv id.$$ Let $\tau,\tau_1,\tau_2$ some parameters to be chosen later as functions of $\epsilon$. The reader is encouraged to consider $1\gg \tau_1\gg \tau_2\gg \tau.$ By Theorem \ref{Positive Proportion}, we can find some interval $I_0$ centered at $\rho_2$ such that $$\lambda(I_0\cap B_{\epsilon}(\rho_2)\cap E^{\eta}(\rho_2))\geq (1-\tau)\lambda(I_0).$$

    Now, $f_{i,N,\theta,u}(I_0)$ is an interval that might not contain $\rho_i.$ To remedy this problem and keep applying Theorem \ref{Positive Proportion} we define $e_{i,N,\theta,u}(I_0)$ to be the smallest interval containing $f_{i,N,\theta,u}(I_0)$ and $\rho_i.$ For $u$ sufficiently small in terms of $\epsilon$ and $\tau$, we can achieve $$(1+\tau)\lambda(f_{i,N,\theta,u}(I_0))\geq \lambda(e_{i,N,\theta,u}(I_0)).$$We now construct an algorithm, that could be thought of us a ``shifted version" of that constructed in Theorem \ref{algorithm}. We will call an interval $e_{i,N,\theta,u}(I)$ containing $\rho_i$ of Type I or Type II according to whether it satisfies the first or second conclusion of Theorem \ref{Positive Proportion} respectively, with respect to $\epsilon$ and the parameters $\tau_1,\tau_2.$

    \begin{itemize}
     \item Set $L_0=\emptyset.$ If $e_{3,N,\theta,u}(I_0)$ is of Type I, then consider $e_{4,N,\theta,u}(I_0)$. If $e_{4,N,\theta,u}(I_0)$ is of Type I, then consider $e_{5,N,\theta,u}(I_0).$ If all of them up to $e_{k,N,\theta,u}(I_0)$ are Type I, terminate the procedure.
     \item Suppose $i_0$ is the smallest number in $\{3,...,k\}$ with $e_{i_0,N,\theta,u}(I_0)$ of Type II. We can then find, by Proposition \ref{Positive Proportion}, some subinterval $I'_1\subset e_{i_0,N,\theta,u}(I_0)$, such that $$\lambda(I'_1)\geq C\tau_1\epsilon\lambda(e_{i_0,N,\theta,u}(I'_0)) \text{ and } \lambda(I'_1\cap B_{\epsilon}(\rho_{i_0})\cap E^{\eta}(\rho_{i_0}))\geq (1-\tau_2)\lambda(I'_1),$$where $C$ is some absolute constant. Define $I''_1=I'_1\cap f_{i_0,N,\theta,u}(I_0).$ We have $$\lambda(I''_1)\geq  \left(C\tau_1\epsilon-\tau\right)\lambda(f_{i_0,N,\theta,u}(I_0)) \text{ and } \lambda(I''_1\cap B_{\epsilon}(\rho_{i_0})\cap E^{\eta}(\rho_{i_0}))\geq (1-\tau_2-\frac{\tau}{C\tau_1\epsilon-\tau})\lambda(I''_1).$$ For $I_1=f_{i_0,N,\theta,u}^{-1}(I''_1),$ we have: \begin{align*}
         \lambda(I_1\cap B_{\epsilon}(\rho_2)\cap E^{\eta}(\rho_2))&\geq \lambda(I_1)-\lambda(I_0\cap \left(B_{\epsilon}(\rho_2)\cap E^{\eta}(\rho_2)\right)^c)\geq \\&\geq\lambda(I_1)-\tau\lambda(I_0)\geq (1-C'\frac{\tau}{C\tau_1\epsilon-\tau})\lambda(I_1),\end{align*} where $C'$ is again some absolute constant depending only on the size of the derivatives of $f_{i_0,N,\theta,u}$. Define $L_1=L_0\cup\{i_0\}$.
     \item We repeat the first step with $I_1$ instead of $I_0$, skipping the numbers in $L_1$ and so on.
     \item Terminate the algorithm at $\mathbb{T}-$th step if $L_{\mathbb{T}}=\{2...,k\}.$
 \end{itemize}
Notice that after at most $k$ steps we obtain an interval $I_f\subset I_0$, some subset $L\subset\{3,...,k\}$ and an absolute constant, which we denote again by $C$ such that: \begin{enumerate}
     \item For all $i\in L$: $$\lambda(f_{i,N,\theta,u}(I_f)\cap B_{\epsilon}(\rho_i)\cap E^{\eta}(\rho_i))\geq (1-(C\tau_1\epsilon-\tau)^{-k}(\tau_2+\frac{\tau}{C\tau_1\epsilon-\tau}))\lambda(f_{i,N,\theta,u}(I_f)).$$
     \item For all $i\in\{3,...,k\}\setminus L$: $$\lambda(f_{i,N,\theta,u}(I_f))\cap B_{\epsilon}(\rho_i)\cap E^{\eta}(\rho_i))\geq (1-\tau_1-\tau)\lambda(f_{i,N,\theta,u}(I_f)).$$
 \end{enumerate}

 Choosing $\tau,\tau_1,\tau_2$ rapidly decaying functions of $\epsilon$ accordingly and defining $$U=\left\{T\in\Sigma_{N,\theta,v}:\, T(\rho_2)\in I_f\right\},$$
concludes the proof.  

\end{proof}
In Theorem \ref{Density theorem} we used diagonal perturbations of $\epsilon-$almost extremal lattices to generate spectrum of full measure. We now show how to further perturb by lattices on the curve $\Sigma_{N,\theta,v}$ to show that the values $$m(P\circ T_{\epsilon}\circ T_u\circ \Delta_{\theta}),$$ where $T_{\epsilon}\in\text{AEL}(\epsilon),$ $T_u\in\Sigma_{N,\theta,v},$ $\theta\in I_N$ cover the entire spectral interval.
\begin{proof}[Proof of Theorem \ref{Main Theorem}:] Fix $\epsilon>0.$ We can assume that $P$ is extremal with respect to $\Z^2$, since we can always achieve that by conjugating with some lattice. By Theorem \ref{Flexibility of the extremal lattice}, we can find some $T_{\epsilon}\in\SL_2(\R)$, such that: 

\begin{itemize}
    \item $m(P\circ T_{\epsilon})>m(P)-\epsilon,$
    \item $\vert\vert T_{\epsilon}-id_2\vert\vert<\epsilon.$ 
    \item $T_{\epsilon}(\rho_i)\in E_2$ for $i\in\{1,2,...,k\}.$
\end{itemize}
We will abuse notation and denote the roots of $P\circ T_{\epsilon}$ by $\rho_i$. This will hopefully not cause any confusion, since throughout this proof we will be considering the roots $P\circ T_{\epsilon}$ and not those of $P$. We can further assume that $$\pdv{F_i}{\rho'_2}\vert_{(\rho_2,\rho_i)}\neq 0$$ for all $i\in\{1,2,...,k\}$ since $\pdv{F_i}{\rho'_2}$ is not indentically $0$. Let $M$ as in the statement of Proposition \ref{Main Theorem measure} corresponding to the form $P\circ T_{\epsilon}$. Applying Theorem \ref{Sigma subvariety} for some $\eta<\frac{4}{5}$, we can find some $\delta>0$, such that for $N$ large enough, $\theta\in I_N$ and $u$ satisfying $$\vert v-\prod_{i\geq2}(\rho_i-\rho_1)\vert<\delta,$$ there exists some open subset $U\subset \Sigma_{N,\theta,v}$ such that $$\mu\left(U\bigcap_{i\geq2}\Pi_i^{-1}(E^{\eta}(\rho_i)\cap B_{\epsilon}(\rho_i))\right)\geq (1-\epsilon)\mu\left(U\right).$$ Now, the set of $\rho\in B(\rho_i,\epsilon)$, for which we can find $(x_c,y_c)\in\mathbb{Z}\times\mathbb{N}$ satisfying $y_c<2MQ_N^{\frac{5}{4}}$ and $$\left\vert\frac{x_c}{y_c}-\frac{\rho}{\rho_1}\frac{P_N}{Q_N}\right\vert<\frac{2M}{(y_cQ_N)^{\frac{5}{4}}},$$ has measure bounded above by $$\sum\limits_{\substack{\frac{X}{Y}\in\mathbb{Q}\cap B(\rho_i,2\epsilon), \\  Y < 2MQ_N^{\frac{5}{4}}}}\frac{4M}{{(YQ_N)}^{\frac{5}{4}}}\ll \epsilon\sum\limits_{Y<2MQ_N^{\frac{5}{4}}}\frac{1}{Y^{\frac{1}{4}}Q_N^{\frac{5}{4}}}\ll\epsilon\phi^{-N\frac{5}{16}}=o_N(1)\epsilon.$$ Similarly, the set of $\rho\in B(\rho_i,\epsilon)$ for which we can find $(x_c,y_c)\in\mathbb{Z}\times\mathbb{N}$ satisfying $y_c>\frac{1}{2M}Q_N^{\frac{5}{4}}$ and $$\left\vert\frac{x_c}{y_c}-\theta{\rho}\right\vert<\frac{2M}{y_c^{n}},$$ is bounded above by $o_N(1)\epsilon.$ \\ 

Hence, putting everything together, for every $u\in\left(\prod_{i\geq2}(\rho_i-\rho_1)-\delta,\prod_{i\geq2}(\rho_i-\rho_1)+\delta\right)$, $N$ large, and $\theta\in I_N$, we can find some $T_u\in\SL_2(\R)$ such that \begin{enumerate}
    \item $T_u(\rho_1)=\rho_1$
    \item $\prod_{i\geq2}(\frac{P_N}{Q_N}-\theta T_{u}(\rho_i))=v$
    \item $$\left\vert\frac{x_c}{y_c}-\frac{T_u(\rho_i)}{\rho_1}\frac{P_N}{Q_N}\right\vert>\frac{2M}{(y_cQ_N)^{\frac{5}{4}}},$$ for all $y_c<2MQ_N^{\frac{5}{4}}$ and $i\in\{2,...,k\}$.
    \item $m(P\circ T_{\epsilon}\circ T_{u})\geq m(P\circ T_{\epsilon})-\epsilon\geq m(P)-2\epsilon.$
    \item $T_{u}(\rho_i)\in E^{\eta}(\rho_i).$
    \item  $$\left\vert\frac{x_c}{y_c}-\theta{T_u(\rho_i)}\right\vert>\frac{2M}{y_c^{n}},$$ for $y_c>\frac{1}{2M}Q_N^{\frac{5}{4}}$ and $i\in\{2,3,...,k\}$.
\end{enumerate}

This $T_u$ of course depends on $N$ and $\theta$ as well. This is however understood, and we omit the extra notation. Assume now $\theta$ is such that \begin{align}\left\vert\frac{x_c}{y_c}-\theta{\rho_1}\right\vert>\frac{2M}{y_c^{n}},\label{theta inequality}\end{align} for $y_c>\frac{1}{2M}Q_N^{\frac{5}{4}}$. By Theorem \ref{Main Theorem measure}, the only possibility for $m(P\circ T_{\epsilon}\circ T_u\circ \Delta_{\theta})$ is that $P\circ T_{\epsilon}\circ T_u\circ \Delta_{\theta}$ is minimized by $(P_N,Q_N)$ or some $(X,Y)$ such that \begin{align*}\left\vert\frac{X}{Y}-{T_u(\rho_i})\right\vert<\frac{2M}{Y^{2+\frac{4}{5}}},\end{align*} for some $i\in\{1,2,...,k\}$. Here, we have used again that $M$ varies continuously as a function of $(\rho_1,...,\rho_k).$ However, since $T_{u}(\rho_i)\in E^{\eta}(\rho_i)$ and $\rho_i\in E_2$,  $X,Y$ should be bounded above by an absolute constant uniform over all $u$ and $\theta$. We have shown that for $N$ large enough and $\theta\in I_N$ satisfying (\ref{theta inequality}), we have $$m(P\circ T_{\epsilon}\circ T_u\circ \Delta_{\theta})=P\circ T_{\epsilon}\circ T_u\circ \Delta_{\theta}(P_N,Q_N)=\theta^{-\frac{n}{2}}(\frac{P_N}{Q_N}-\theta\rho_1)u.$$ or $$m(P\circ T_{\epsilon}\circ T_u\circ \Delta_{\theta})=P\circ T_{\epsilon}\circ T_u\circ \Delta_{\theta}(X,Y),$$ where $X,Y$ are bounded by some absolute constant. For $N\rightarrow\infty,$ and $\epsilon\rightarrow 0$ we know that $$P\circ T_{\epsilon}\circ T_u\circ \Delta_{\theta}(X,Y)\rightarrow P(X,Y)\geq m(P),$$ and therefore for any $t_0>0$, if we choose $\theta$ to satisfy $\frac{\theta-\theta_N}{\frac{P_N}{\rho_1Q_N}-\theta_N}\geq t_0,$ we have that for $N$ large enough and $\epsilon$ small enough, $$m(P\circ T_{\epsilon}\circ T_u\circ \Delta_{\theta})=\theta^{-\frac{n}{2}}(\frac{P_N}{Q_N}-\theta\rho_1)u.$$ To sum up, we have constructed an interval in the spectrum of the form $$\left\{\theta^{-\frac{n}{2}}\left\vert(\frac{P_N}{Q_N}-\theta\rho_1)u\right\vert, u\in\left(\prod_{i\geq2}(\rho_i-\rho_1)-\delta,\prod_{i\geq2}(\rho_i-\rho_1)+\delta\right)\right\}.$$ in $\Spec(P)$. It is easy now to see that, since condition (\ref{theta inequality}) is satisfied for all but $o_N(1)$ proportion of $I_N$, these intervals cover all of $(0,M_P)$ by taking $N\rightarrow\infty,$ $\epsilon\rightarrow 0.$ By Lemma \ref{existence of extremal lattice}, we have $M_P\in \Spec(P)$. As for $0,$ we can find some lattice $\Lambda$, such that $P\circ\Lambda$ has a rational root and hence $0$ is also in $\Spec(P).$ This concludes the proof of Theorem \ref{Main Theorem}.
\end{proof}

\section*{Acknowledgements}
I am greatly indebted to my advisor Peter Sarnak for many enlightening discussions and his constant encouragement throughout the execution of this project. I also want to express my gratitude towards Alex Gamburd, Elon Lindenstrauss, Nikolai Moshchevitin, Uri Shapira, Barak Weiss and Nina Zubrilina for interesting conversations in connection to this problem as well as comments and suggestions on earlier drafts of this paper.

\bibliographystyle{unsrtnat}
\bibliography{References}

\begin{thebibliography}{14}
\providecommand{\natexlab}[1]{#1}
\providecommand{\url}[1]{\texttt{#1}}
\expandafter\ifx\csname urlstyle\endcsname\relax
  \providecommand{\doi}[1]{doi: #1}\else
  \providecommand{\doi}{doi: \begingroup \urlstyle{rm}\Url}\fi

\bibitem[Mahler(1946{\natexlab{a}})]{Mahler}
Kurt Mahler.
\newblock On lattice points in n-dimensional star bodies {I}. existence theorems.
\newblock \emph{Proceedings of the Royal Society of London. Series A. Mathematical and Physical Sciences}, 187\penalty0 (1009):\penalty0 151--187, 1946{\natexlab{a}}.

\bibitem[Mahler(1946{\natexlab{b}})]{mahlerII}
Kurt Mahler.
\newblock Lattice points in n-dimensional star bodies. {II}.
\newblock In \emph{Proc. Acad. Amsterdam}, volume~49, pages 624--532, 1946{\natexlab{b}}.

\bibitem[Sarnak(2022-2023)]{Sarnak}
Peter Sarnak.
\newblock {Chern Lectures Berkeley - ``Prescribing the Spectra of Locally Uniform Geometries"}.
\newblock 2022-2023.
\newblock URL \url{https://publications.ias.edu/sarnak/paper/2728}.

\bibitem[Markoff(1879)]{Markov}
Andrey Markoff.
\newblock Sur les formes quadratiques binaires ind{\'e}finies.
\newblock \emph{Mathematische Annalen}, 15\penalty0 (3-4):\penalty0 381--406, 1879.

\bibitem[Hall(1947)]{Hall}
Marshall Hall.
\newblock On the sum and product of continued fractions.
\newblock \emph{Annals of Mathematics}, pages 966--993, 1947.

\bibitem[Freiman(1975)]{Freiman}
GA~Freiman.
\newblock Diofantovy priblizheniya i geometriya chisel (zadacha markova).
\newblock \emph{Diophantine approximations and the geometry of numbers (Markov’s problem)] Kalinin. Gosudarstv. Univ., Kalinin}, 1975.

\bibitem[Mordell(1945)]{Mordell}
L.~J. Mordell.
\newblock {On Numbers Represented by Binary Cubic Forms}.
\newblock \emph{Proceedings of the London Mathematical Society}, s2-48\penalty0 (1):\penalty0 198--228, 01 1945.
\newblock ISSN 0024-6115.
\newblock \doi{10.1112/plms/s2-48.1.198}.

\bibitem[Davenport(1941)]{DavenportonMordell}
H~Davenport.
\newblock On a conjecture of {M}ordell concerning binary cubic forms.
\newblock In \emph{Mathematical Proceedings of the Cambridge Philosophical Society}, volume~37, pages 325--330. Cambridge University Press, 1941.

\bibitem[Cassels(1997)]{CasselsBook}
J.~W.~S. Cassels.
\newblock \emph{An introduction to the geometry of numbers}.
\newblock Classics in Mathematics. Springer-Verlag, Berlin, 1997.
\newblock ISBN 3-540-61788-4.
\newblock Corrected reprint of the 1971 edition.

\bibitem[Cusick and Flahive(1989)]{CusickFlahive}
Thomas~W. Cusick and Mary~E. Flahive.
\newblock \emph{The {M}arkoff and {L}agrange spectra}, volume~30 of \emph{Mathematical Surveys and Monographs}.
\newblock American Mathematical Society, Providence, RI, 1989.
\newblock ISBN 0-8218-1531-8.
\newblock \doi{10.1090/surv/030}.

\bibitem[Khinchin(1997)]{KhinchinTextbook}
A.~Ya. Khinchin.
\newblock \emph{Continued fractions}.
\newblock Dover Publications, Inc., Mineola, NY, russian edition, 1997.
\newblock ISBN 0-486-69630-8.
\newblock With a preface by B. V. Gnedenko, Reprint of the 1964 translation.

\bibitem[Springer(1977)]{InvariantTheory}
T.~A. Springer.
\newblock \emph{Invariant theory}, volume Vol. 585 of \emph{Lecture Notes in Mathematics}.
\newblock Springer-Verlag, Berlin-New York, 1977.

\bibitem[Ra{\l}owski et~al.(2009)Ra{\l}owski, Szczepaniak, and {\.Z}eberski]{Steinhaus}
Robert Ra{\l}owski, Przemys{\l}aw Szczepaniak, and Szymon {\.Z}eberski.
\newblock A generalization of steinhaus' theorem and some nonmeasurable sets.
\newblock 2009.

\bibitem[Erdos and Oxtoby(1955)]{erdos1955partitions}
Paul Erdos and John~C Oxtoby.
\newblock Partitions of the plane into sets having positive measure in every non-null measurable product set.
\newblock \emph{Transactions of the American Mathematical Society}, 79\penalty0 (1):\penalty0 91--102, 1955.

\end{thebibliography}
\end{document}